\documentclass[a4paper,12pt]{article}
\usepackage[T2A]{fontenc}                      
\usepackage[cp1251]{inputenc}           
\usepackage[all]{xy}
\usepackage{amssymb}
\usepackage{cite}
\usepackage{cmap}
\usepackage{latexsym}
\usepackage{enumerate}
\usepackage{amsmath, amsthm, amscd, amsfonts, amssymb, graphicx, color}
\usepackage[left=2.5cm,right=2.5cm,top=2cm,bottom=2cm]{geometry}
\usepackage{indentfirst}
\usepackage{array}
\usepackage{bm}
\usepackage{float}

\usepackage{wrapfig}

\usepackage{authblk}

\newtheorem{corollary}{Corollary}
\newtheorem{conjecture}{Conjecture}

\newtheorem{theorem}{Theorem}
\newtheorem{lemma}{Lemma}
\newtheorem{definition}{Definition}

\title{On infinite families of $[P_n]$-irregular graphs}

\author[1]{Tatiana~Dovzhenok\thanks{Corresponding author. E-mail: \texttt{t.dovzhenok@mail.ru}}}
\author[2,3]{Ilya~Lukashenko}
\author[3]{Andrei~Mikhalev}
\author[3]{Yahor~Filiuta}

\affil[1]{\small Research Laboratory ``Mathematics of Hybrid Intelligence Systems'', \
	
	Francisk Skorina Gomel State University, Gomel, 246028, Belarus}
\affil[2]{Faculty of Information Technologies and Control, Belarusian State \ 
	
	University of Informatics and Radioelectronics, Minsk, 220013, Belarus}
\affil[3]{Research Laboratory ``Algebra and Geometry of Complex Systems'', \
	
	Francisk Skorina Gomel State University, Gomel, 246028, Belarus}

\date{} 

\begin{document}	
\maketitle

\begin{abstract}
	This paper presents the first systematic study of $[F]$-irregular graphs, a concept that parallels classical $F$-irregularity. For a fixed graph $F$, a graph $G$ is \mbox{$[F]$-irregular} if the 
	numbers of its induced subgraphs isomorphic to $F$ containing a given vertex are pairwise distinct for all vertices of~$G$. We prove that there exist infinitely many $[P_n]$-irregular graphs for any path $P_n$ of order $n \ge 3$. 
	We establish that a non-trivial $[P_3]$-irregular graph of order $k$ exists if and only if $k \ge 7$. 
	Finally, we propose the Strong Conjecture on $[F]$-irregular graphs.
	
\end{abstract}
	
\section{Introduction}
While regular graphs are central to algebraic and structural graph theory, vertex-distinguishing problems have motivated multiple paradigms of irregularity~\cite{gtwa}. In what follows, we exclusively consider simple, finite, and undirected graphs. Since no non-trivial graph can have pairwise distinct vertex degrees, a core challenge is to define generalized degree concepts that distinguish all vertices of the graph.

One such direction originates from the seminal work of Chartrand, Holbert, Oellermann, and Swart~\cite{r1}, who associated each vertex $v$ of a graph $G$ with the number of its subgraphs isomorphic to a fixed graph $F$ that contain $v$. This quantity, termed the $F$-degree of $v$, naturally captures structural asymmetry: a~graph $G$ is called $F$-irregular if its vertices have mutually distinct $F$-degrees. In~\cite{r1}, the authors put forward a fundamental conjecture on the existence of such graphs, motivated by their study of the cases where $F$ is a complete graph or a star.

\begin{conjecture}[Chartrand et al.~\cite{r1}, 1987]\label{con1}
	For every connected graph $F$ of order $|F| \ge 3$, there exists a non-trivial $F$-irregular graph.
\end{conjecture}

Further progress for broader classes of graphs was made only in 2024, when Dovzhenok, Filuta, and Chuhai~\cite{r2} demonstrated that infinite families of $F$-irregular graphs exist for all $2$-connected graphs $F$ with minimum degree $2$. In the same work, the authors proposed a significantly stronger version of Conjecture~\ref{con1}.

\begin{conjecture}[Strong Conjecture about $F$-irregular graphs, Dovzhenok et al.~\cite{r2}, 2024]\label{con2}
	For every connected graph $F$ of order $|F| \ge 3$, there are infinitely many $F$-irregular graphs.
\end{conjecture}

Subsequently, Dovzhenok verified Conjecture~\ref{con1} for paths~\cite{r3} and confirmed Conjecture~\ref{con2} for all graphs $F$ of diameter~2~\cite{r4}.

\smallskip
In the present paper, in contrast to the classical setting based on ordinary subgraphs, we focus on irregularity through the lens of \emph{induced} subgraphs. This approach was originally suggested by Ali, Chartrand, and Zhang in~\cite{gtwa}.

\begin{definition}
	Let $F$ and $G$ be graphs. The \emph{induced $F$-degree} (or simply \emph{$[F]$-degree}) of a vertex $v$ in $G$ is the number of induced subgraphs of $G$ that are isomorphic to $F$ and contain $v$. A graph $G$ is \emph{$[F]$-irregular} if the $[F]$-degrees of  its vertices are pairwise distinct.
\end{definition}

Unlike ordinary subgraphs, which are preserved under edge addition, their induced counterparts are highly sensitive to any such modification. This structural fragility complicates the construction of $[F]$-irregular graphs compared to the classical case.

As far as we are aware, this work initiates the systematic investigation of $[F]$-irregularity, focusing on the canonical class of paths. Specifically, we completely solve the existence problem for $[P_n]$-irregular graphs posed in~\cite{r3}, and fully determine the possible orders of non-trivial $[P_3]$-irregular graphs. Our main contributions are summarized in the following theorems.

\begin{theorem}\label{t1}
	For any path $P_n$ of order $n \ge 3$, there exist infinitely many $[P_n]$-irregular graphs.
\end{theorem}

\begin{theorem}\label{t2}
	For an integer $k$, a non-trivial $[P_3]$-irregular graph of order $k$ exists if and only if $k \ge 7$.
\end{theorem}

Our approach is constructive: for each $n \ge 3$, we provide explicit families of graphs, determine the exact values of all $[P_n]$-degrees, and show that they are pairwise distinct.

The remainder of this paper is organized as follows. Section 2 introduces the necessary preliminaries and notation. Section 3 is devoted to the proof of Theorem~\ref{t2}, characterizing $[P_3]$-irregular graphs. Sections 4 and 5 focus on $[P_4]$- and $[P_5]$-irregularity, respectively. In Section 6, we settle the general case for all $n \ge 6$, thereby completing the proof of Theorem~\ref{t1}. Finally, Section 7 outlines the Strong Conjecture on $[F]$-irregular graphs.

\section{Preliminaries}
All numerical parameters in this paper are integers. Any set of indexed elements of the form \mbox{$\bigl\{\alpha_x \bigm\vert i \le x \le j\bigr\}$} is also written as \mbox{$\{\alpha_i, \dots, \alpha_j\}$}. For a graph $G$, let $V(G)$ and $E(G)$ be its vertex and edge sets, respectively. The \textit{order} of $G$ is $|G| = |V(G)|$, and $G$ is \textit{non-trivial} if $|G| \ge 2$. An edge incident to vertices $u$ and $v$ is denoted by $uv$. The \textit{degree} of a vertex $v \in V(G)$ is the number of edges incident to $v$, and $G$ is \textit{regular} if all its vertices have the same degree. We use the notation \mbox{$K_m$} for the complete graph of order~$m$.
A path $P_n$ of order $n$ is a graph with \mbox{$V(P_n) = \{u_1, \dots, u_n\}$} and \mbox{$E(P_n) = \bigl\{u_iu_{i+1} \bigm\vert 1 \le i \le n-1\bigr\}$}. We write such a path as the vertex sequence \mbox{$u_1 u_2 \dots u_n$}, where $u_1$ and $u_n$ are called the \textit{end-vertices}. 
If a graph contains a unique path with the end-vertices $u$ and $v$, this path is denoted by \mbox{$u \dots v$}; in particular, \mbox{$u \dots u$} represents the single vertex $u$. A~graph is \textit{connected} if every pair of its vertices is joined by a path. 

A subgraph $H$ of $G$ is \textit{induced} if \mbox{$E(H) = \bigl\{uv \in E(G) \bigm\vert u, v \in V(H)\bigr\}$}. For a vertex $v \in V(G)$, we denote by \mbox{$\mathcal{P}_n(G,v)$} the family of all induced subgraphs of $G$ isomorphic to $P_n$ and containing $v$. Consequently, the $[P_n]$-degree of a vertex $v$ in $G$ is exactly \mbox{$|\mathcal{P}_n(G,v)|$}.

\section{On the order of $[P_3]$-irregular graphs}

\subsection{Construction of $[P_3]$-irregular graphs of odd order $k \ge 9$}
We begin by introducing a parameterized family of graphs $\bigl\{G(m,3)\bigr\}_{m \ge 4}$, where each graph has a corresponding odd order $k = 2m + 1$.

\begin{definition}
	Let $m \ge 4$ be an integer. The graph $G(m,3)$, schematically shown in Figure~1, is specified by its vertex and edge sets as follows:
\vspace{-0.2cm}	
\begin{align*}
	V \bigl(G(m,3)\bigr) &= \{a_1, \dots, a_m\} \cup \{b_1,  \dots, b_m\} \cup \{x\}, \\[1mm]
	E \bigl(G(m,3)\bigr) &= \bigl\{ xa_i \bigm\vert 1 \le i \le m-1 \bigr\} \cup \bigl\{ a_i b_j \bigm\vert 1 \le j \le i \le m \bigr\}.
\end{align*}
\end{definition}
\vspace{-0.8cm}
\begin{figure}[ht]
	\centering
	\includegraphics[width=8cm]{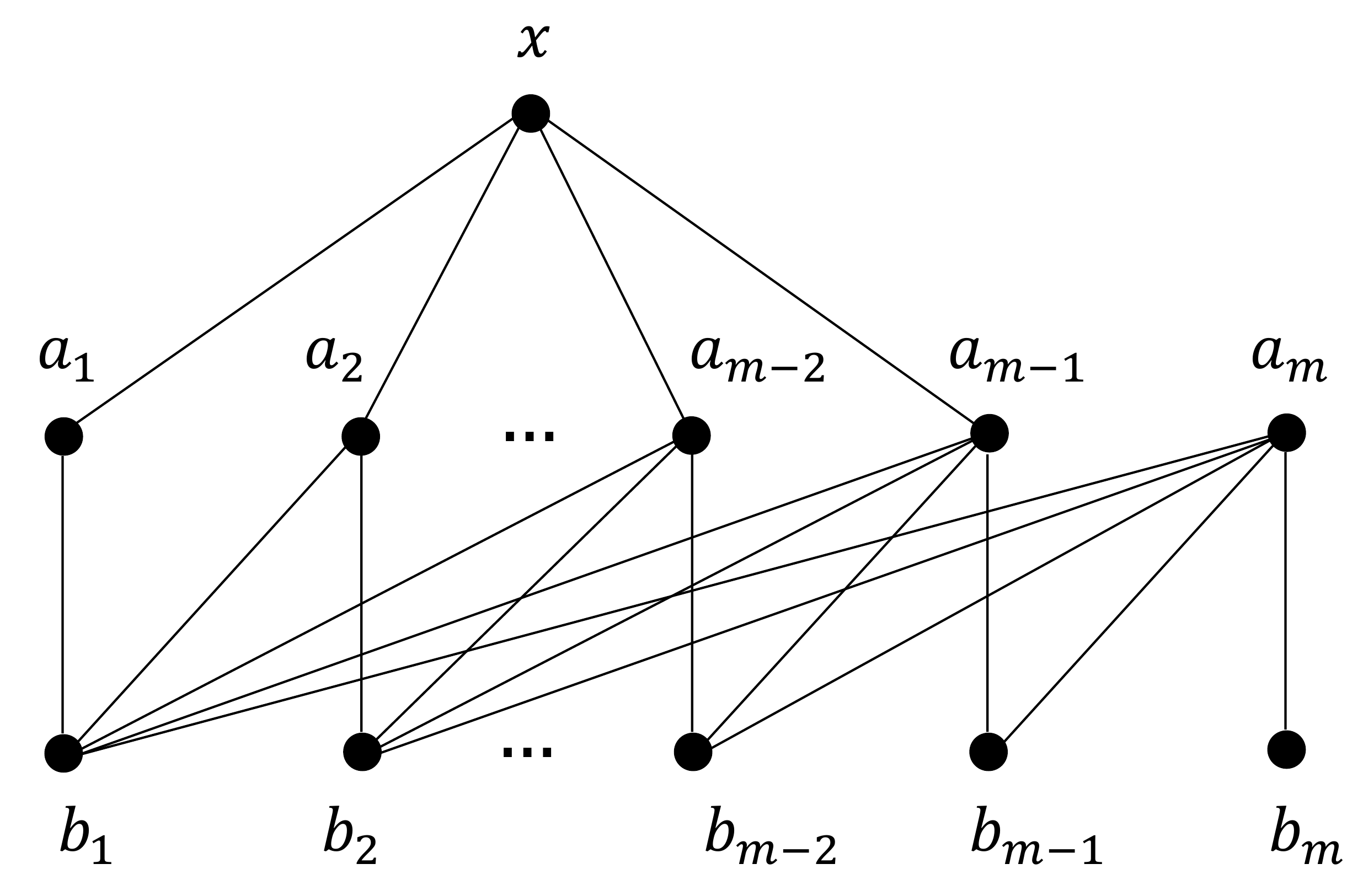}
	\caption{\small Graph $G(m,3)$.}
	\label{fig1}
\end{figure}

For the graph $G(m,3)$, let $X$, $A_i$, and $B_j$ denote the $[P_3]$-degrees of the vertices $x$, $a_i$, and $b_j$, respectively, where $1 \le i, j \le m$.

\begin{lemma}\label{l1}
	The vertex $[P_3]$-degrees in $G(m,3)$ are given by
	\begin{align*}
		X &= (m-1)^2, \qquad A_m = m(m-1), \\[1mm]
		A_i &= m(i + 1) - 2 \quad \text{for } 1 \le i \le m-1, \\[1mm]
		B_j &= m(m-j+1) - 1 \quad \text{for } 1 \le j \le m.
	\end{align*}
\end{lemma}

\begin{proof}
Clearly, the $[P_3]$-degree of a vertex $v$ in $G(m,3)$ is obtained by summing the cardinalities of disjoint subfamilies into which $\mathcal{P}_3(G(m,3), v)$ decomposes.
\medskip

\noindent \textit{Computation of $X$.} For the vertex $x$, this decomposition takes the form $\mathcal{X}_1 \sqcup \mathcal{X}_2$ where
\begin{align*}
	\mathcal{X}_1 &= \bigl\{ x a_i b_j \bigm\vert 1 \le j \le i \le m-1 \bigr\}, & \bigl|\mathcal{X}_1\bigr| &= \frac{1}{2}m(m-1); \\[1mm]
	\mathcal{X}_2 &= \bigl\{ a_i x a_k \bigm\vert 1 \le i < k \le m-1 \bigr\}, & \bigl|\mathcal{X}_2\bigr| &= \frac{1}{2}(m-1)(m-2).
\end{align*}
Hence, $X = |\mathcal{X}_1| + |\mathcal{X}_2| = (m-1)^2$.
\medskip 

\noindent \textit{Computation of $A_i$.} For $i=m$, the family $\mathcal{P}_3(G(m,3), a_m)$ splits into $\mathcal{A}_{m,1} \sqcup \mathcal{A}_{m,2}$ where
\begin{align*}
	\mathcal{A}_{m,1} &= \bigl\{ a_i b_j a_m \bigm\vert 1 \le j \le i \le m-1 \bigr\}, & \bigl|\mathcal{A}_{m,1}\bigr| &= \frac{1}{2}m(m-1); \\[1mm]
	\mathcal{A}_{m,2} &= \bigl\{ b_j a_m b_k \bigm\vert 1 \le j < k \le m \bigr\}, & \bigl|\mathcal{A}_{m,2}\bigr| &= \frac{1}{2}m(m-1).
\end{align*}
Consequently, $A_m = |\mathcal{A}_{m,1}| + |\mathcal{A}_{m,2}| = m(m-1)$.
\medskip

\noindent Now fix $1 \le i \le m-1$. The family $\mathcal{P}_3(G(m,3), a_i)$ partitions into $\bigsqcup_{t=1}^{4} \mathcal{A}_{i,t}$ where 
\begin{align*}
	\mathcal{A}_{i,1} &= \bigl\{ a_i x a_k \bigm\vert 1 \le k \le m-1, \, k \neq i \bigr\}, & \bigl|\mathcal{A}_{i,1}\bigr| &= m-2; \\
	\mathcal{A}_{i,2} &= \bigl\{ a_i b_j a_k \bigm\vert 1 \le j \le i, \, j \le k \le m, \, k \neq i \bigr\}, & \bigl|\mathcal{A}_{i,2}\bigr| &= \sum_{j=1}^{i}(m-j) = \frac{1}{2}i(2m-i-1); \\
	\mathcal{A}_{i,3} &= \bigl\{ x a_i b_j \bigm\vert 1 \le j \le i \bigr\}, & \bigl|\mathcal{A}_{i,3}\bigr| &= i; \\
	\mathcal{A}_{i,4} &= \bigl\{ b_j a_i b_k \bigm\vert 1 \le j < k \le i \bigr\}, & \bigl|\mathcal{A}_{i,4}\bigr| &= \frac{1}{2}i(i-1).
\end{align*}
Straightforward calculation shows that $A_i = \sum_{t=1}^{4} |\mathcal{A}_{i,t}| = m(i+1)-2$.
\bigskip

\noindent \textit{Computation of $B_j$.} For any $1 \le j \le m$, we have $\mathcal{P}_3(G(m,3), b_j)=\bigsqcup_{t=1}^{3} \mathcal{B}_{j,t}$ where
\begin{align*}
	\mathcal{B}_{j,1} &= \bigl\{ x a_i b_j \bigm\vert j \le i \le m-1 \bigr\}, \quad
	\mathcal{B}_{j,2} = \bigl\{ a_i b_j a_k \bigm\vert j \le i < k \le m \bigr\}, \\[1mm]
	\mathcal{B}_{j,3} &= \bigl\{ b_j a_i b_k \bigm\vert j \le i \le m, \, 1 \le k \le i, \, k \neq j \bigr\},
\end{align*}
with the respective cardinalities 
\begin{align*}								
	&|\mathcal{B}_{j,1}| = m-j, \qquad  |\mathcal{B}_{j,2}| = \frac{1}{2}(m-j+1)(m-j), \\
	&|\mathcal{B}_{j,3}| = \sum_{i=j}^{m}(i-1) = \frac{1}{2}(m-j+1)(m+j-2). 
\end{align*}
Summing these values yields $B_j = \sum_{t=1}^{3} |\mathcal{B}_{j,t}| = m(m-j+1)-1$.	
\end{proof}

\begin{theorem}\label{t3}
	For every integer $m \ge 4$, the graph $G(m,3)$ is $[P_3]$-irregular.
\end{theorem}

\begin{proof}
	Fix an integer $m \ge 4$. Modulo $m$, the vertex $[P_3]$-degrees in $G(m,3)$ obtained in Lemma~\ref{l1} split into four disjoint congruence classes
	\begin{align*}
		X &\equiv 1 \pmod m, \\
		A_m &\equiv 0 \pmod m, \\
		A_i &\equiv m-2 \pmod m \quad \text{for } 1 \le i \le m-1, \\[1mm]
		B_j &\equiv m-1 \pmod m \quad \text{for } 1 \le j \le m.
	\end{align*}	
	This partition and the strict monotonicity of the sequences $\bigl\{A_i\bigr\}_{i=1}^{m-1}$ and $\bigl\{B_j\bigr\}_{j=1}^m$ imply that all these $[P_3]$-degrees are pairwise distinct. Thus, $G(m,3)$ is $[P_3]$-irregular.
\end{proof}

\begin{corollary}\label{c1}
	There exist infinitely many $[P_3]$-irregular graphs.
\end{corollary}

\subsection{Proof of Theorem~\ref{t2}}

Recall that Theorem~\ref{t2} states that a non-trivial \mbox{$[P_3]$-irregular} graph of order $k$ exists if and only if $k \ge 7$.

\begin{proof}
\noindent\textbf{Necessity.} To begin with, every $[P_3]$-irregular graph is inherently asymmetric, i.e., 
its automorphism group is trivial. According to~\cite{r5}, there are no 
asymmetric graphs of orders $2 \le k \le 5$, and there are exactly eight such graphs of order $6$, as shown in Figure~\ref{fig2}. A direct check reveals that none of them is $[P_3]$-irregular; there, in each graph, two vertices with the same $[P_3]$-degree are marked with their equal values. Consequently, no $[P_3]$-irregular graph exists 
for any order $2 \le k \le 6$. The necessity is proved.
	
\begin{figure}[ht]
	\centering
	\includegraphics[width=11cm]{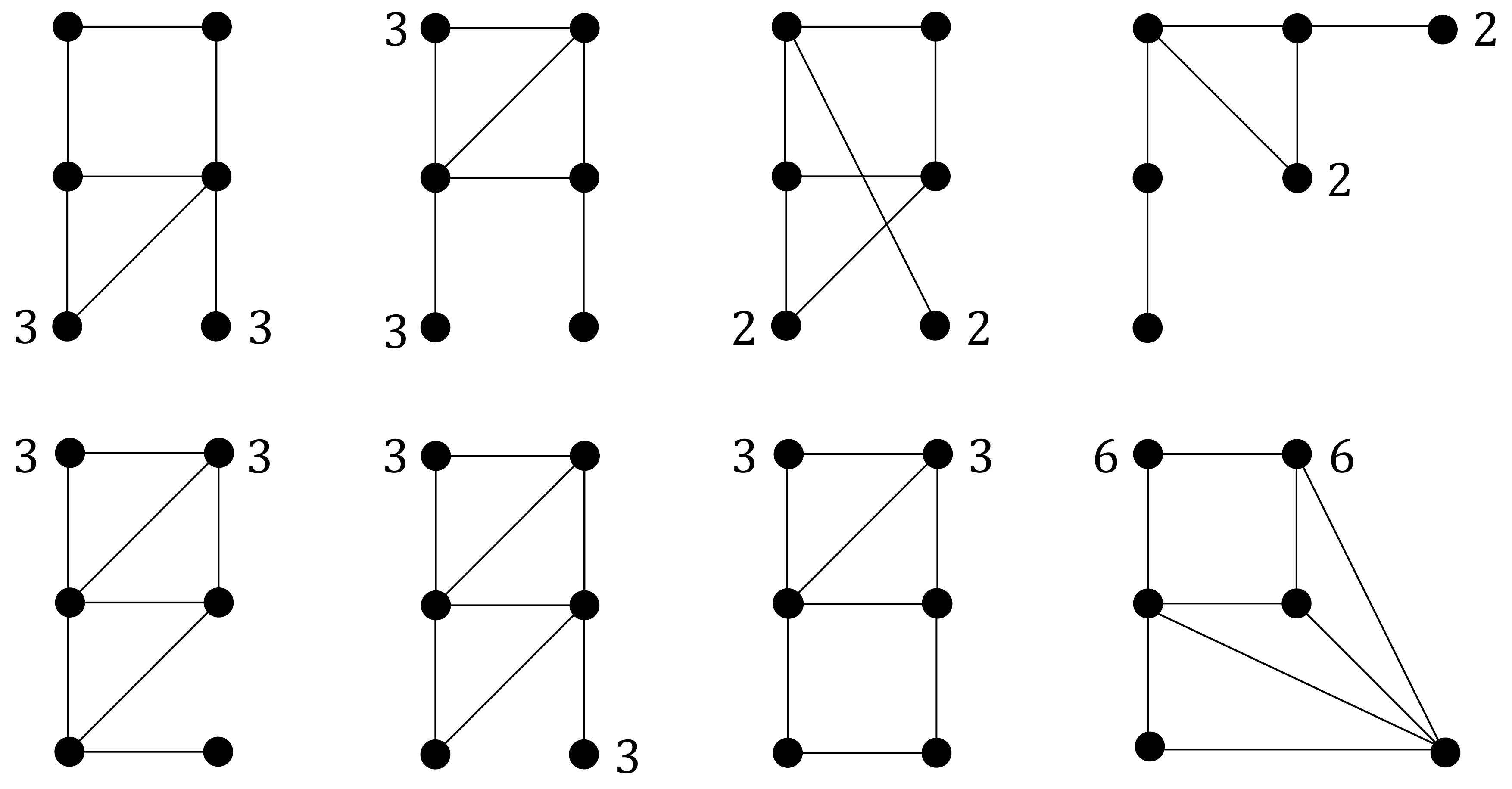}
	\caption{\small The eight asymmetric graphs of order $6$, in each of which two vertices are numerically labeled with their equal $[P_3]$-degree values.}
	\label{fig2}
\end{figure}
			
	\medskip 
	\noindent\textbf{Sufficiency.} By Theorem~\ref{t3}, for each integer $m \ge 4$, the graph $G(m,3)$ of odd order $k = 2m+1 \ge 9$ is $[P_3]$-irregular, and the exact formulae in Lemma~\ref{l1} ensure that all its vertex $[P_3]$-degrees are strictly positive. To complete the odd orders, Figure~\ref{fig3} depicts the $[P_3]$-irregular graph $G(3,3)$ of order $7$. It contains exactly $16$ induced subgraphs isomorphic to $P_3$: 
	$v_1v_2v_3$, $v_1v_2v_6$, $v_1v_2v_7$, $v_1v_5v_3$, $v_1v_5v_6$; 
	$v_2v_1v_5$, $v_2v_3v_4$, $v_2v_3v_5$, $v_2v_6v_5$; 
	$v_3v_2v_6$, $v_3v_5v_6$, $v_3v_7v_6$; 
	$v_4v_3v_5$, $v_4v_3v_7$; $v_5v_3v_7$, and $v_5v_6v_7$. 
	Consequently, the $[P_3]$-degrees of the vertices $v_1, \dots, v_7$ in $G(3,3)$ are $6, 8, 10, 3, 9, 7$, and $5$, respectively, which are also strictly positive.
	\begin{figure}[ht]
		\centering
		\includegraphics[width=4cm]{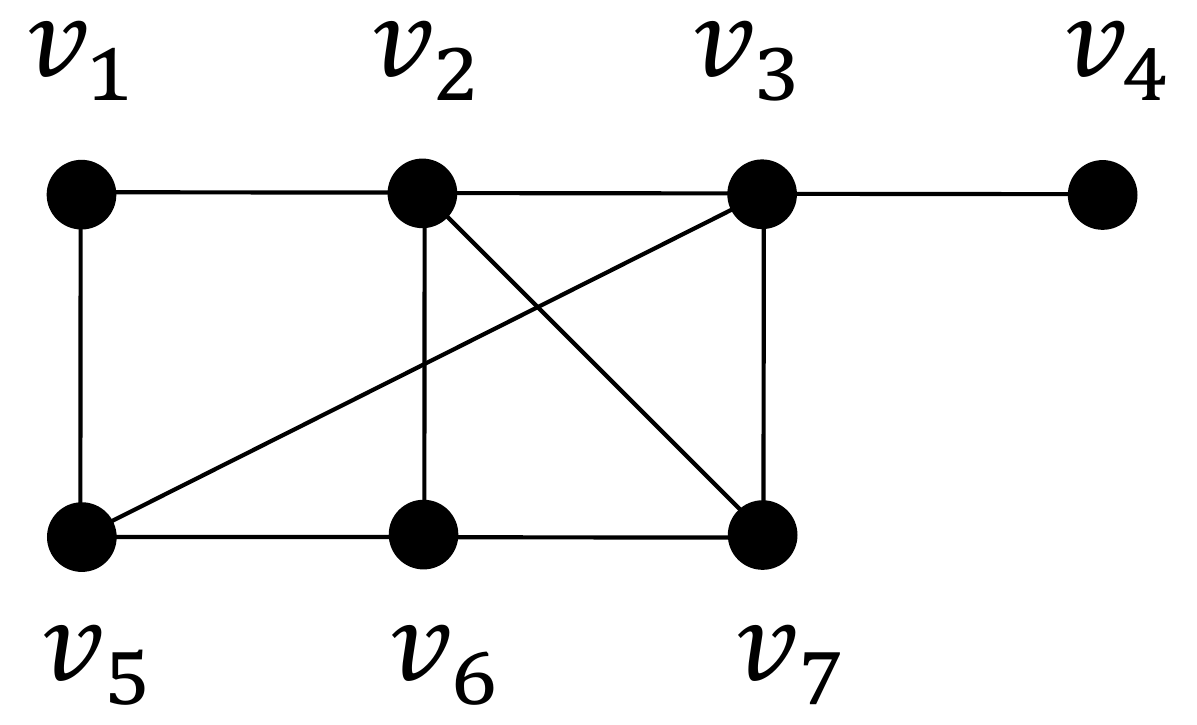}
		\caption{\small $[P_3]$-irregular graph $G(3,3)$ of order $7$.}
		\label{fig3}
	\end{figure}

A $[P_3]$-irregular graph of every even order $k \ge 8$ can be constructed by adjoining an isolated vertex to $G\left(\frac{1}{2}(k-2),3\right)$, thereby completing the proof of sufficiency.
\end{proof}

\section{$[P_4]$-irregular graphs}
In this section, we construct \mbox{$[P_4]$-irregular} graphs of every order $k \ge 17$, focusing first on the case of odd $k$.

\begin{definition}
	Let $m \ge 7$ be an integer. The graph $G(m,4)$ of odd order $2m+3$, depicted in Figure~\ref{fig4}, is characterized by the vertex and edge sets
	\begin{align*}
		V \bigl(G(m,4)\bigr) &= \{a_1, \dots, a_m\} \cup \{b_1, \dots, b_m\} \cup \{x_1, x_2, y_1\}, \\[1mm]
		E \bigl(G(m,4)\bigr) &= \bigl\{ a_i a_j \bigm\vert 1 \le i < j \le m \bigr\} \cup \bigl\{ a_i b_j \bigm\vert 1 \le j \le i \le m \bigr\} \\[1mm]
		&\quad \cup \bigl\{ x_2 a_i \bigm\vert 1 \le i \le m \bigr\} \cup \{x_1 x_2, \, b_1 y_1, \, b_m y_1\}.
	\end{align*}
\end{definition}

\begin{figure}[ht]
	\centering
	\includegraphics[width=7cm]{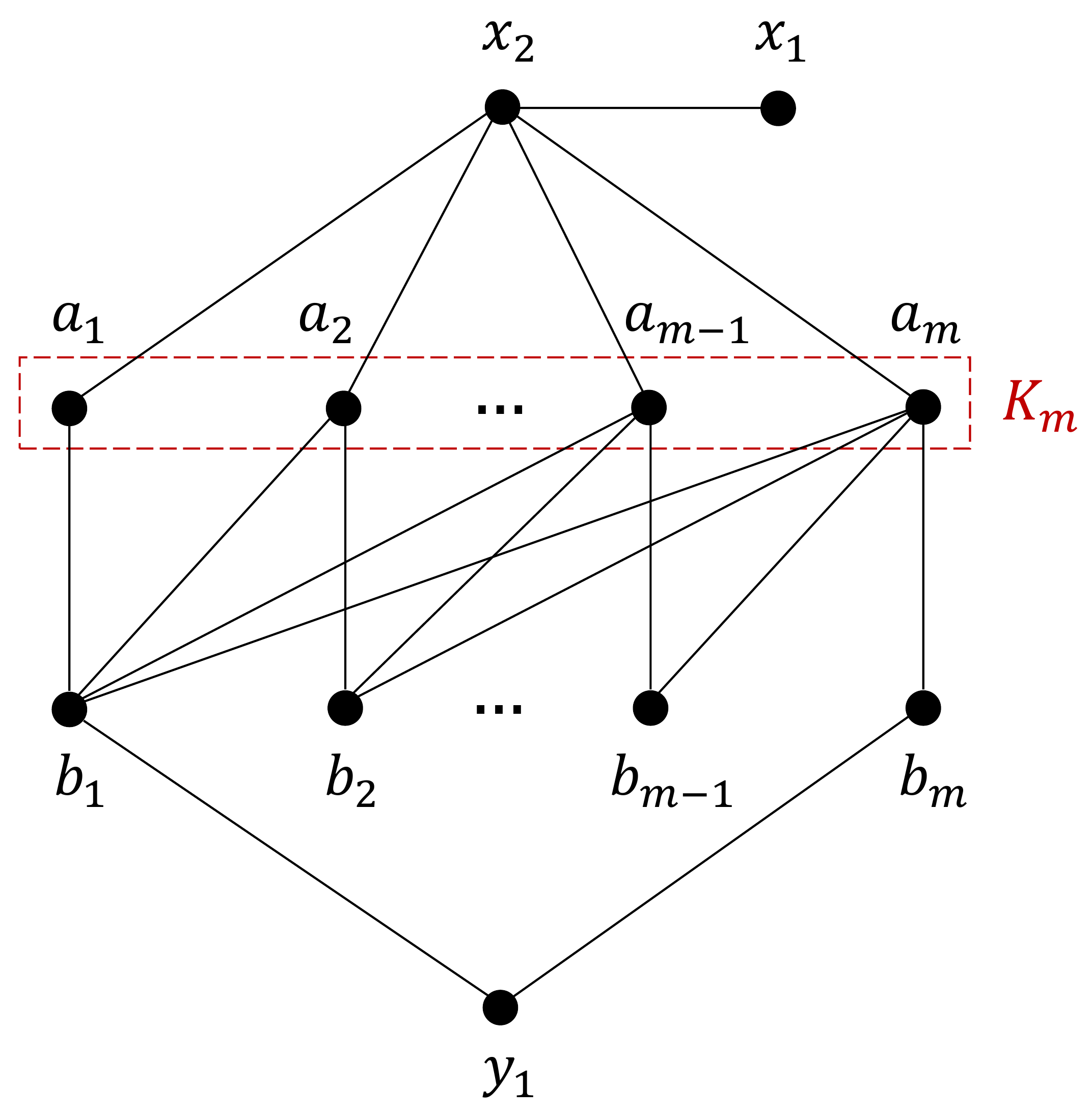}
	\caption{\small Graph $G(m,4)$.}
	\label{fig4}
\end{figure}

For the graph $G(m,4)$, let $X_1$, $X_2$, $A_i$, $B_j$, and $Y_1$ denote the $[P_4]$-degrees of the vertices $x_1$, $x_2$, $a_i$, $b_j$, and $y_1$, respectively, where $1 \le i, j \le m$.

\begin{lemma}\label{l2}
	The vertex $[P_4]$-degrees in $G(m,4)$ are given by
	\begin{align*}
		X_1 &= \frac{1}{2}m(m+1), &
		X_2 &= \frac{1}{2}m(m+1) + m + 1, \\[1mm]
		A_i &= 2i + 2 \quad \text{for } 1 \le i \le m - 1, &
		A_m &= 4m - 3, \\[1mm]
		B_j &= 2m - 2j + 3 \quad \text{for } 2 \le j \le m - 1, &
		B_1 &= \frac{1}{2}m(m+1) + 2m - 2, \\[1mm]
		B_m &= 3m - 2, &
		Y_1 &= \frac{1}{2}m(m+1) + 3m - 4.
	\end{align*}
\end{lemma}

\begin{proof}
The $[P_4]$-degrees in $G(m,4)$ are determined analogously to the proof of Lemma~\ref{l1} using $U = \{a_1, \dots, a_{m-1}\} \cup \{b_2, \dots, b_{m-1}\} \cup \{x_2\}$.
\medskip

\noindent \textit{Computation of $X_1, X_2$.} For $x_1$, we have $\mathcal{P}_4(G(m,4), x_1) = \bigl\{x_1x_2a_ib_j \bigm\vert 1 \le j \le i \le m \bigr\}$, which yields $X_1 = \displaystyle \frac{1}{2}m(m+1)$. 
\medskip
	
\noindent For $x_2$, $\mathcal{P}_4(G(m,4), x_2)$ partitions into $\bigsqcup_{t=1}^{3} \mathcal{X}_{2,t}$, where
	\begin{align*}
		\mathcal{X}_{2,1} &= \mathcal{P}_4(G(m,4), x_1), & |\mathcal{X}_{2,1}| &= X_1=\frac{1}{2}m(m+1); \\
		\mathcal{X}_{2,2} &= \bigl\{x_2a_ib_1y_1 \bigm\vert 1 \le i \le m\bigr\}, & |\mathcal{X}_{2,2}| &= m; \\[1mm]
		\mathcal{X}_{2,3} &= \{x_2a_mb_my_1\}, & |\mathcal{X}_{2,3}| &= 1.
	\end{align*}
	Consequently, $X_2 = \sum_{t=1}^{3} |\mathcal{X}_{2,t}| = \displaystyle\frac{1}{2}m(m+1) + m + 1$.
	\bigskip
	
	\noindent \textit{Computation of $A_i$.} Fix $1 \le i \le m-1$. Then $\mathcal{P}_4(G(m,4), a_i)=\bigsqcup_{t=1}^{3} \mathcal{A}_{i,t}$, where
	\begin{align*}
		\mathcal{A}_{i,1} &= \bigl\{x_1x_2a_ib_j \bigm\vert 1 \le j \le i\bigr\}, & |\mathcal{A}_{i,1}| &= i; \\[1mm]
		\mathcal{A}_{i,2} &= \bigl\{b_j a_i b_1 y_1 \bigm\vert 2 \le j \le i\bigr\}, & |\mathcal{A}_{i,2}| &= i - 1; \\[1mm]
		\mathcal{A}_{i,3} &= \{x_2a_ib_1y_1, \, a_ib_1y_1b_m, \, a_ia_mb_my_1\}, & |\mathcal{A}_{i,3}| &= 3. 
	\end{align*}
	It follows that $A_i = \sum_{t=1}^{3} |\mathcal{A}_{i,t}| = 2i + 2$.
	\bigskip
	
	\noindent For $i = m$, the family $\mathcal{P}_4(G(m,4), a_m)$ decomposes into $\bigsqcup_{t=1}^{3} \mathcal{A}_{m,t}$, where
	\begin{align*}
		\mathcal{A}_{m,1} &= \bigl\{x_1x_2a_mb_j \bigm\vert 1 \le j \le m\bigr\}, & |\mathcal{A}_{m,1}| &= m; \\[1mm]
		\mathcal{A}_{m,2} &= \bigl\{v a_m b_1 y_1 \bigm\vert v \in \{x_2\} \cup \{b_2, \dots, b_{m-1}\}\bigr\}, & |\mathcal{A}_{m,2}| &= m - 1; \\[1mm]
		\mathcal{A}_{m,3} &= \bigl\{u a_m b_m y_1 \bigm\vert u \in U\bigr\}, & |\mathcal{A}_{m,3}| &= 2m - 2.
	\end{align*}
	This leads to $A_m = \sum_{t=1}^{3} |\mathcal{A}_{m,t}| = 4m - 3$.
	\bigskip
	
	\noindent \textit{Computation of $B_j$.} For $j = 1$, we obtain $\mathcal{P}_4(G(m,4), b_1) = \bigsqcup_{t=1}^{4} \mathcal{B}_{1,t}$, where
	\begin{align*}
		\mathcal{B}_{1,1} &= \bigl\{x_1x_2a_ib_1 \bigm\vert 1 \le i \le m\bigr\}, & |\mathcal{B}_{1,1}| &= m; \\[1mm]
		\mathcal{B}_{1,2} &= \mathcal{X}_{2,2}, & |\mathcal{B}_{1,2}| &= m; \\[1mm]
		\mathcal{B}_{1,3} &= \bigl\{a_ib_1y_1b_m \bigm\vert 1 \le i \le m-1\bigr\}, & |\mathcal{B}_{1,3}| &= m-1; \\
		\mathcal{B}_{1,4} &= \bigl\{b_j a_i b_1 y_1 \bigm\vert 2 \le j \le i \le m, \, j \neq m\bigr\}, & |\mathcal{B}_{1,4}| &= \displaystyle\frac{1}{2}m(m-1)-1.
	\end{align*}
	Thus, $B_1 = \sum_{t=1}^{4} |\mathcal{B}_{1,t}| = \displaystyle\frac{1}{2}m(m+1) + 2m - 2$.
	\bigskip
	
	\noindent Fix $2 \le j \le m-1$. We express $\mathcal{P}_4(G(m,4), b_j)$ as the disjoint union $\bigsqcup_{t=1}^{3} \mathcal{B}_{j,t}$, where
	\begin{align*}
		\mathcal{B}_{j,1} &= \bigl\{x_1x_2a_ib_j \bigm\vert j \le i \le m\bigr\}, & |\mathcal{B}_{j,1}| &= m - j + 1; \\[1mm]
		\mathcal{B}_{j,2} &= \bigl\{b_ja_ib_1y_1 \bigm\vert j \le i \le m\bigr\}, & |\mathcal{B}_{j,2}| &= m - j + 1; \\[1mm]
		\mathcal{B}_{j,3} &= \{b_ja_mb_my_1\}, & |\mathcal{B}_{j,3}| &= 1.
	\end{align*}
	Hence, $B_j = \sum_{t=1}^{3} |\mathcal{B}_{j,t}| = 2m - 2j + 3$.
	\bigskip
	
	\noindent Finally, in the case $j = m$, the relation $\mathcal{P}_4(G(m,4), b_m) = \{x_1x_2a_mb_m\} \sqcup \mathcal{A}_{m,3} \sqcup \mathcal{B}_{1,3}$ implies $B_m = 1 + (2m - 2) + (m - 1) = 3m - 2$. 
	\bigskip
	
	\noindent \textit{Computation of $Y_1$.} Since $\mathcal{P}_4(G(m,4), y_1) = \mathcal{X}_{2,2} \sqcup \mathcal{A}_{m,3} \sqcup \mathcal{B}_{1,3} \sqcup \mathcal{B}_{1,4}$, we conclude that
	\[Y_1 = m + (2m - 2) + (m - 1) + \frac{1}{2}m(m-1)-1 = \frac{1}{2}m(m+1) + 3m - 4. \qedhere \]
\end{proof}

\begin{theorem}\label{t4}
	For every integer $m \ge 7$, the graph $G(m,4)$ is $[P_4]$-irregular.
\end{theorem}
\begin{proof}
	Fix an integer $m \ge 7$. By Lemma~\ref{l2}, the sequence $\bigl\{A_i\bigr\}_{i=1}^{m-1}$ is strictly increasing with a maximum of $A_{m-1} = 2m$ and contains only even values, whereas $\bigl\{B_j\bigr\}_{j=2}^{m-1}$ is strictly decreasing with a maximum of $B_2 = 2m-1$ and comprises only odd values, ensuring that their terms are pairwise distinct. Moreover, we have the following chain of inequalities
	\[
	2m < B_m < A_m < X_1 < X_2 < B_1 < Y_1.
	\]
	These arguments collectively demonstrate that all vertex $[P_4]$-degrees in $G(m,4)$ are pairwise distinct, meaning that $G(m,4)$ is $[P_4]$-irregular.
\end{proof}

\begin{corollary}\label{c2}
	For every integer $k \ge 17$, there exists a $[P_4]$-irregular graph of order $k$.
\end{corollary}
\begin{proof}
	For odd orders, the existence follows directly from Theorem~\ref{t4}. For any even order $k \ge 18$, the desired graph is obtained by adding an isolated vertex to the \mbox{$[P_4]$-irregular} graph $G\bigl(\frac{1}{2}(k-4), 4\bigr)$ of odd order $k-1$ having strictly positive vertex  $[P_4]$-degrees by Lemma~\ref{l2}.
\end{proof}

\section{$[P_5]$-irregular graphs}
This section presents a construction of \mbox{$[P_5]$-irregular} graphs of every order $k \ge 19$, starting with odd values of $k$.
\begin{definition}
	Let $m \ge 7$ be an integer. The graph $G(m,5)$ of odd order $2m+5$, illustrated in Figure~\ref{fig5}, is defined by
	\begin{align*}
		V \bigl(G(m,5)\bigr) &= \{a_1, \dots, a_m\} \cup \{b_1, \dots, b_m\} \cup \{x_1, x_2, x_3, y_1, y_2\}, \\[1mm]
		E \bigl(G(m,5)\bigr) &= \bigl\{ a_i a_j \bigm\vert 1 \le i < j \le m \bigr\} \cup \bigl\{ a_i b_j \bigm\vert 1 \le j \le i \le m \bigr\} \\[1mm]
		&\quad \cup \bigl\{ x_3 a_i \bigm\vert 1 \le i \le m \bigr\} \cup \{x_1 x_2, \, x_2 x_3, \, b_1 y_2, \, b_m y_2, \, y_1 y_2\}.
	\end{align*}
\end{definition}
\vspace{-0.5cm}

\begin{figure}[ht]
	\centering
	\includegraphics[width=7cm]{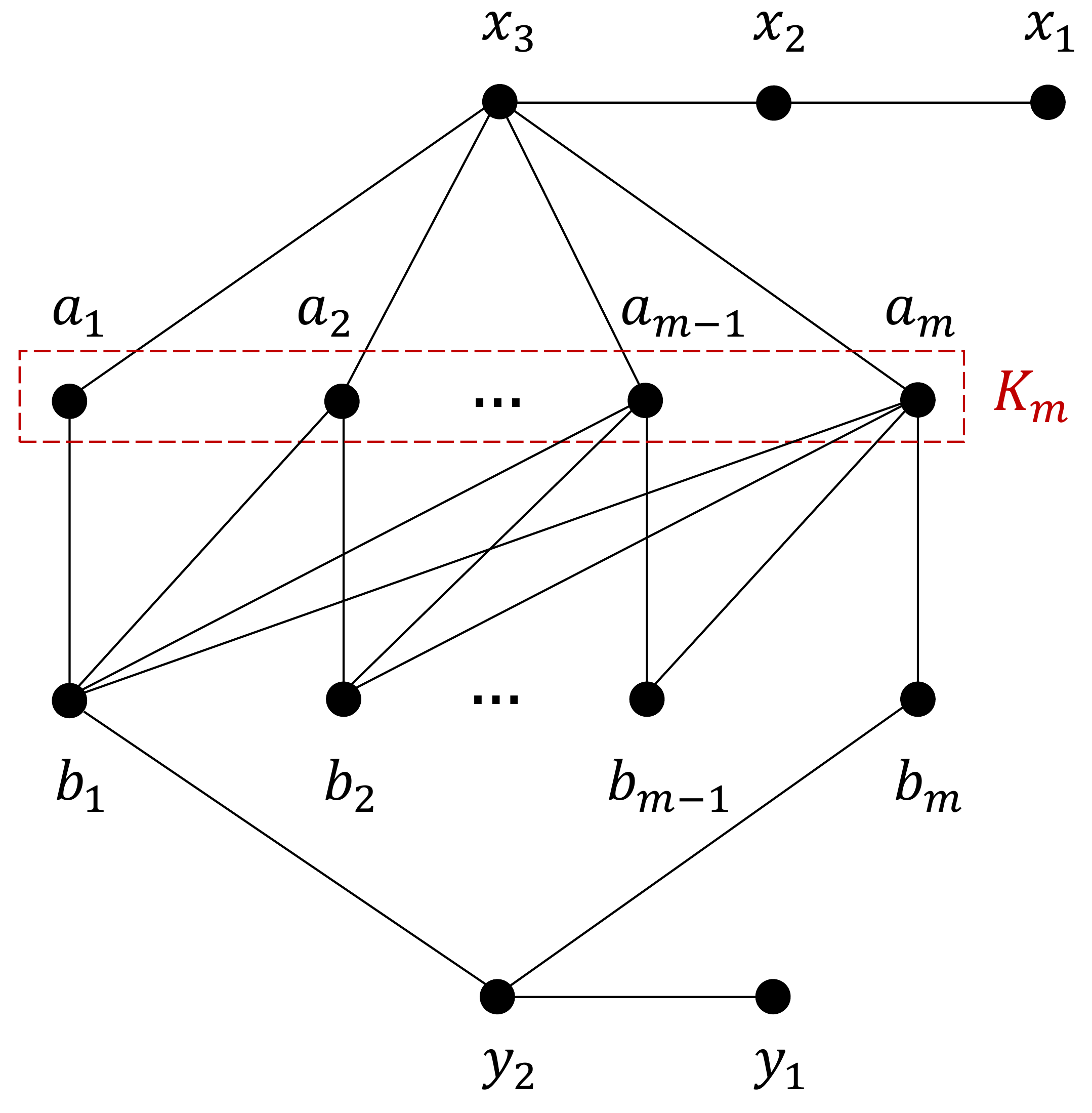}
	\caption{\small Graph $G(m,5)$.}
	\label{fig5}
\end{figure}

For the graph $G(m,5)$, let $X_1$, $X_2$, $X_3$, $A_i$, $B_j$, $Y_1$, and $Y_2$ be the $[P_5]$-degrees of the vertices $x_1$, $x_2$, $x_3$, $a_i$, $b_j$, $y_1$, and $y_2$, respectively, where $1 \le i, j \le m$.

\begin{lemma}\label{l3}
	The vertex $[P_5]$-degrees in $G(m,5)$ are given by
	\begin{align*}
		X_1 &= \frac{1}{2}m(m+1), &
		X_2 &= \frac{1}{2}m(m+1) + m + 1, \\
		X_3 &= \frac{1}{2}m(m+1) + 3m + 1, &
		A_i &= 3i + 2 \quad \text{for } 1 \le i \le m - 1, \\
		A_m &= 4m - 1, &
		B_1 &= m^2 + 2m - 1, \\
		B_j &= 3m - 3j + 3 \quad \text{for } 2 \le j \le m - 1, &
		B_m &= \frac{1}{2}m(m+1) + m, \\
		Y_1 &= \frac{1}{2}m(m+1) + 2m - 3, &
		Y_2 &= m^2 + 3m - 2.
	\end{align*}
\end{lemma}

\begin{proof} Let $U = \{a_1, a_2, \dots, a_{m-1}\} \cup \{b_2, b_3, \dots, b_{m-1}\} \cup \{x_3\}$. 
\medskip
	
\noindent \textit{Computation of $X_1$, $X_2$, and $X_3$.} For $x_1$, $\mathcal{P}_5(G(m,5), x_1) = \{x_1x_2x_3a_i b_j \mid 1 \le j \le i \le m \}$, yielding $X_1 = \bigl|\mathcal{P}_5(G(m,5), x_1)\bigr| = \displaystyle\frac{1}{2}m(m+1)$. 
\medskip

\noindent For $x_2$, $\mathcal{P}_5(G(m,5), x_2) = \bigsqcup_{t=1}^{3} \mathcal{X}_{2,t}$ where
	\begin{align*}
		\mathcal{X}_{2,1} &= \mathcal{P}_5(G(m,5), x_1), & |\mathcal{X}_{2,1}| &= X_1= \frac{1}{2}m(m+1); \\
		\mathcal{X}_{2,2} &= \bigl\{x_2x_3a_i b_1y_2 \bigm\vert 1 \le i \le m\bigr\}, & |\mathcal{X}_{2,2}| &= m; \\[1mm]
		\mathcal{X}_{2,3} &= \{x_2x_3a_mb_my_2\}, & |\mathcal{X}_{2,3}| &= 1.
	\end{align*}
Then $X_2 = \sum_{t=1}^{3} |\mathcal{X}_{2,t}| = \displaystyle\frac{1}{2}m(m+1) + m + 1$.
\medskip
	
\noindent For $x_3$, $\mathcal{P}_5(G(m,5), x_3) = \bigsqcup_{t=1}^{4} \mathcal{X}_{3,t}$ where
	\begin{align*}
		\mathcal{X}_{3,1} &= \mathcal{P}_5(G(m,5), x_2), & |\mathcal{X}_{3,1}| &= X_2=\frac{1}{2}m(m+1) + m + 1; \\
		\mathcal{X}_{3,2} &= \bigl\{x_3a_ib_1y_2b_m \bigm\vert 1 \le i \le m-1\bigr\}, & |\mathcal{X}_{3,2}| &= m-1; \\[1mm]
		\mathcal{X}_{3,3} &= \bigl\{x_3a_ib_1y_2y_1 \bigm\vert 1 \le i \le m\bigr\}, & |\mathcal{X}_{3,3}| &= m; \\[1mm]
		\mathcal{X}_{3,4} &= \{x_3a_mb_my_2y_1\}, & |\mathcal{X}_{3,4}| &= 1,
	\end{align*}
resulting in $X_3 = \sum_{t=1}^{4} |\mathcal{X}_{3,t}| = \displaystyle\frac{1}{2}m(m+1) + 3m + 1$.
	\medskip
	
	\noindent \textit{Computation of $A_i$.} Now let $1 \le i \le m-1$. Then $\mathcal{P}_5(G(m,5), a_i)=\bigsqcup_{t=1}^{4} \mathcal{A}_{i,t}$ where
	\begin{align*}
		\mathcal{A}_{i,1} &= \bigl\{x_1x_2x_3a_ib_s \bigm\vert 1 \le s \le i\bigr\}, & |\mathcal{A}_{i,1}| &= i; \\[1mm]
		\mathcal{A}_{i,2} &= \{x_2x_3a_ib_1y_2, \, x_3a_ib_1y_2b_m, \, x_3a_ib_1y_2y_1, \, a_ia_mb_my_2y_1\}, & |\mathcal{A}_{i,2}| &= 4; \\[1mm]
		\mathcal{A}_{i,3} &= \bigl\{b_sa_ib_1y_2b_m \bigm\vert 2 \le s \le i\bigr\}, & |\mathcal{A}_{i,3}| &= i - 1; \\[1mm]
		\mathcal{A}_{i,4} &= \bigl\{b_sa_ib_1y_2y_1 \bigm\vert 2 \le s \le i\bigr\}, & |\mathcal{A}_{i,4}| &= i - 1.
	\end{align*}
	This leads to $A_i = \sum_{t=1}^{4} |\mathcal{A}_{i,t}| = 3i + 2$.
	\medskip
	
	\noindent For the boundary case $i = m$, $\mathcal{P}_5(G(m,5), a_m) = \bigsqcup_{t=1}^{4} \mathcal{A}_{m,t}$ where
	\begin{align*}
		\mathcal{A}_{m,1} &= \bigl\{x_1x_2x_3a_mb_j \bigm\vert 1 \le j \le m\bigr\}, & |\mathcal{A}_{m,1}| &= m; \\[1mm]
		\mathcal{A}_{m,2} &= \{x_2x_3a_mb_1y_2, \, x_2x_3a_mb_my_2, \, x_3a_mb_1y_2y_1\}, & |\mathcal{A}_{m,2}| &= 3; \\[1mm]
		\mathcal{A}_{m,3} &= \bigl\{u a_m b_m y_2 y_1 \bigm\vert u \in U \bigr\}, & |\mathcal{A}_{m,3}| &= 2m - 2; \\[1mm]
		\mathcal{A}_{m,4} &= \bigl\{b_sa_mb_1y_2y_1 \bigm\vert 2 \le s \le m-1\bigr\}, & |\mathcal{A}_{m,4}| &= m - 2.
	\end{align*}
	Consequently, $A_m = \sum_{t=1}^{4} |\mathcal{A}_{m,t}| = 4m - 1$.
	\medskip
	
	\noindent \textit{Computation of $B_j$.} For $j = 1$, $\mathcal{P}_5(G(m,5), b_1) = \bigsqcup_{t=1}^{4} \mathcal{B}_{1,t}$ where
	\begin{align*}
		\mathcal{B}_{1,1} &= \bigl\{x_1x_2x_3a_i b_1 \bigm\vert 1 \le i \le m\bigr\}, & |\mathcal{B}_{1,1}| &= m; \\[1mm]
		\mathcal{B}_{1,2} &= \mathcal{X}_{2,2} \sqcup \mathcal{X}_{3,2} \sqcup \mathcal{X}_{3,3}, & |\mathcal{B}_{1,2}| &= 3m - 1; \\[1mm]
		\mathcal{B}_{1,3} &= \bigl\{b_k a_i b_1 y_2 b_m \bigm\vert 2 \le k \le i \le m-1\bigr\}, & |\mathcal{B}_{1,3}| &= \frac{1}{2}(m-1)(m-2); \\[1mm]
		\mathcal{B}_{1,4} &= \bigl\{b_k a_i b_1 y_2 y_1 \bigm\vert 2 \le k \le i \le m, \, k \neq m \bigr\}, & |\mathcal{B}_{1,4}| &= \frac{1}{2}m(m-1)-1.	
	\end{align*}
	This yields $B_1 = \sum_{t=1}^{4} |\mathcal{B}_{1,t}| = m^2 + 2m - 1$.
	\medskip
	
	\noindent Now let $2 \le j \le m-1$. We have $\mathcal{P}_5(G(m,5), b_j) = \bigsqcup_{t=1}^{3} \mathcal{B}_{j,t}$ where
	\begin{align*}
		\mathcal{B}_{j,1} &= \bigl\{x_1x_2x_3a_ib_j, \, b_ja_ib_1y_2y_1 \bigm\vert j \le i \le m\bigr\}, & |\mathcal{B}_{j,1}| &= 2m - 2j + 2; \\[1mm]
		\mathcal{B}_{j,2} &= \bigl\{b_ja_ib_1y_2b_m \bigm\vert j \le i \le m-1\bigr\}, & |\mathcal{B}_{j,2}| &= m - j; \\[1mm]
		\mathcal{B}_{j,3} &= \{b_ja_mb_my_2y_1\}, & |\mathcal{B}_{j,3}| &= 1,
	\end{align*}
	which evaluates to $B_j = \sum_{t=1}^{3} |\mathcal{B}_{j,t}| = 3m - 3j + 3$.
	\medskip
	
	\noindent For $j = m$, $\mathcal{P}_5(G(m,5), b_m) = \{x_1x_2x_3a_mb_m, \, x_2x_3a_mb_my_2 \} \sqcup \mathcal{X}_{3,2} \sqcup \mathcal{A}_{m,3} \sqcup \mathcal{B}_{1,3}$. Hence,   
	\[
	B_m = 2 + (m - 1) + (2m - 2) + \frac{1}{2}(m-1)(m-2) = \frac{1}{2}m(m+1) + m.
	\]
	
	\noindent \textit{Computation of $Y_1$ and $Y_2$.} For $y_1$, $\mathcal{P}_5(G(m,5), y_1) = \mathcal{X}_{3,3} \sqcup \mathcal{B}_{1,4} \sqcup \mathcal{A}_{m,3}$ which provides
	\[
	Y_1 = m + \frac{1}{2}m(m-1)-1 + (2m - 2) = \frac{1}{2}m(m+1) + 2m - 3.
	\]
	Finally, for $y_2$, $\mathcal{P}_5(G(m,5), y_2) = \mathcal{P}_5(G(m,5), y_1) \sqcup \mathcal{X}_{2,2} \sqcup \mathcal{X}_{2,3} \sqcup \mathcal{X}_{3,2} \sqcup \mathcal{B}_{1,3}$. This establishes
	\[
	Y_2 = Y_1 + m + 1 + (m - 1) + \frac{1}{2}(m-1)(m-2) = m^2 + 3m - 2. \qedhere
	\]
\end{proof}

\begin{theorem}\label{t5}
	For every integer $m \ge 7$, the graph $G(m,5)$ is $[P_5]$-irregular.
\end{theorem}
\begin{proof}
	Fix an integer $m \ge 7$. Immediate inspection of Lemma~\ref{l3} yields the strict ordering
	\[
	A_m < X_1 < B_m < X_2 < Y_1 < X_3 < B_1 < Y_2.
	\]
	Next, we observe that the sequence $\bigl\{A_i\bigr\}_{i=1}^{m-1}$ is strictly increasing with a maximum of $A_{m-1} = 3m - 1$, and its terms leave a remainder of $2$ upon division by $3$. In turn, the sequence $\bigl\{B_j\bigr\}_{j=2}^{m-1}$ is strictly decreasing with a maximum of $B_2 = 3m - 3$, and its terms leave a remainder of $0$ upon division by $3$. Thus, the entries of these two sequences are pairwise distinct. Combined with the inequality $B_2 < A_{m-1} < A_m = 4m - 1$, this implies that all the vertex $[P_5]$-degrees in $G(m,5)$ are pairwise distinct, confirming that  $G(m,5)$ is \mbox{$[P_5]$-irregular}.
\end{proof}

Analogously to the proof of Corollary~\ref{c2}, we obtain the following result. 

\begin{corollary}\label{c3}
	For every integer $k \ge 19$, there exists a \mbox{$[P_5]$-irregular} graph of order $k$.
\end{corollary}

\section{$[P_n]$-irregular graphs for $n \ge 6$}

To extend our construction to paths of arbitrary order $n \ge 6$, we introduce a parameterized family of graphs $\{G(m,n)\}$ and establish sufficient conditions on the parameter~$m$ to ensure that every such graph has pairwise distinct vertex $[P_n]$-degrees.

\begin{definition}
	Let $n \ge 6$ and $m \ge 2n$ be integers. The graph $G(m,n)$ (see Figure~\ref{fig6}) is defined as follows
	\begin{align*}
		V\bigl(G(m,n)\bigr) &= \{a_1, \dots, a_m\} \cup \{b_1, \dots, b_m\} \cup \{x_1, \dots, x_{n-2}\} 
		 \cup \{y_1, \dots, y_{n-3}\} \cup \{l\}, \\[1mm]
		E\bigl(G(m,n)\bigr) &= \bigl\{a_ia_j \bigm\vert 1 \le i < j \le m\bigr\} \cup \bigl\{a_ib_j \bigm\vert 1 \le j \le i \le m\bigr\} \\[1mm]
		&\quad \cup \bigl\{x_ix_{i+1} \bigm\vert 1 \le i \le n-3\bigr\} \cup \bigl\{x_{n-2}a_i \bigm\vert 1 \le i \le m\bigr\} \\[1mm]
		&\quad \cup \bigl\{y_iy_{i+1} \bigm\vert 1 \le i \le n-4\bigr\}  \cup \{b_1y_{n-3}, \, b_my_{n-3}\} \cup \bigl\{b_i l \bigm\vert 1 \le i \le m\bigr\}.
	\end{align*}
\end{definition}

\begin{figure}[ht]
	\centering
	\includegraphics[width=9.5cm]{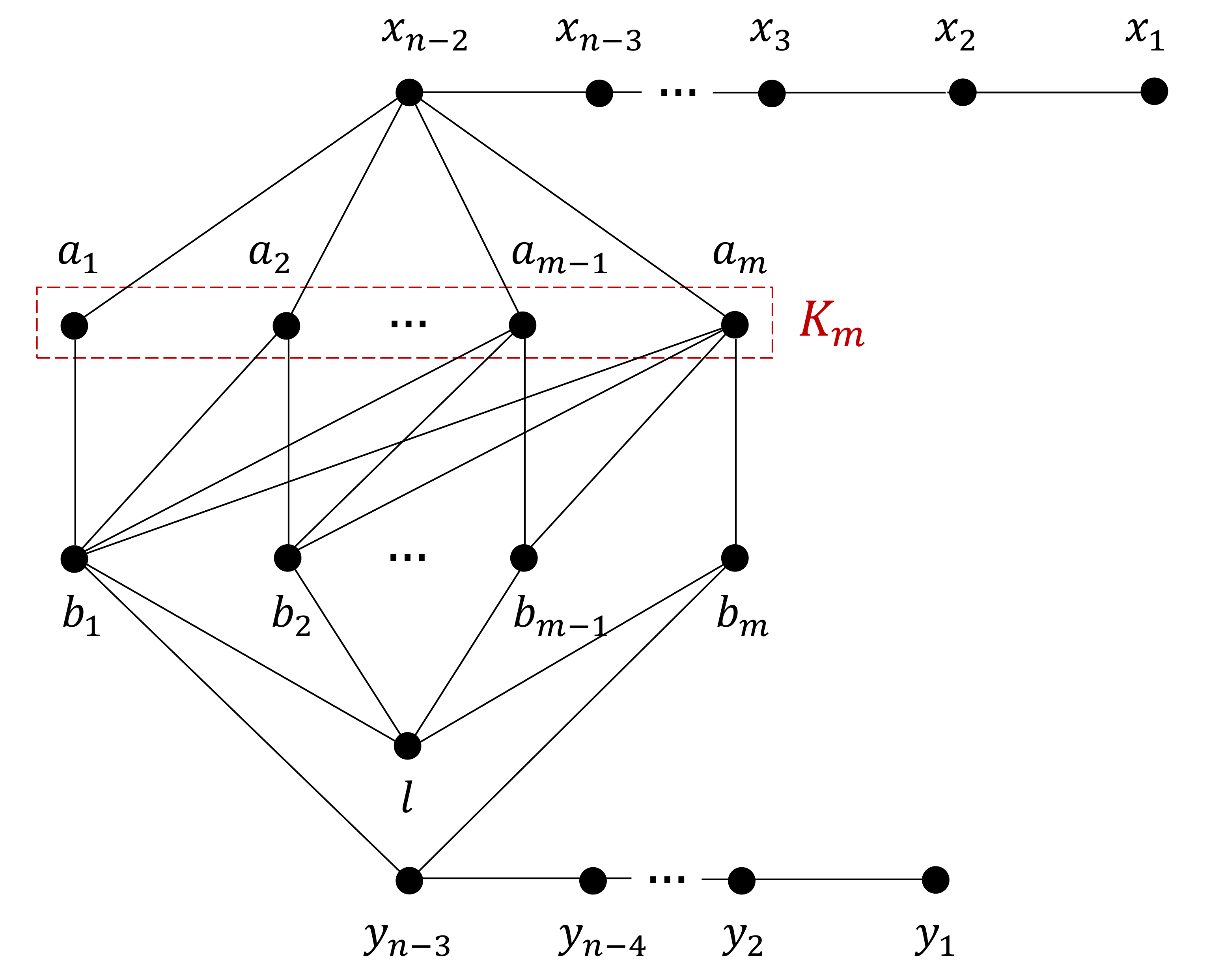}
	\caption{\small Graph $G(m,n)$.}
	\label{fig6}
\end{figure}

For the graph $G(m,n)$, let $A_i$, $B_j$, $X_k$, $Y_s$, and $L$ denote the $[P_n]$-degrees of the vertices $a_i$, $b_j$, $x_k$, $y_s$, and $l$,  respectively, where $1 \le i, j \le m$, $1 \le k \le n-2$, and $1 \le s \le n-3$.

\begin{lemma}\label{l4}
	The $[P_n]$-degrees of the vertices $x_k$ in $G(m,n)$ satisfy
	\begin{align*}
		X_1 &= \frac{1}{2}m(m+1), \qquad X_2 = (m+1)^2, \\[3pt]
		X_k &= \frac{1}{6}m^3 + \frac{1}{2}(k-1)m^2 + \frac{1}{6}(32-3k)m + 2k-5 \quad \text{for } 3 \le k \le n-2.
	\end{align*}
\end{lemma}

\begin{proof}
For each $1 \le k \le n-2$, let $\mathcal{X}_k'$ denote the family of induced paths of order $n$ in $G(m,n)$ with the end-vertex $x_k$. The structure of $G(m,n)$ yields the recurrence
\[
X_1 = |\mathcal{X}_1'|, \quad X_k = X_{k-1} + |\mathcal{X}_k'| \quad \text{for } 2 \le k \le n-2.
\]
	
\noindent\textit{Computation of $X_1$.} Since $\mathcal{X}_1'= \bigl\{ x_1 \dots x_{n-2} a_i b_j \bigm| 1 \le j \le i \le m \bigr\}$, it follows that
\[
X_1 = |\mathcal{X}_1'| = \frac{1}{2}m(m+1).
\]
	
\noindent\textit{Computation of $X_2$.} We decompose the family $\mathcal{X}_2' = \coprod_{t=1}^{3} \mathcal{X}_{2,t}'$, where
\begin{align*}
	\mathcal{X}_{2,1}' &= \bigl\{ x_2 \dots x_{n-2} a_i b_j l \bigm| 1 \le j \le i \le m \bigr\}, & |\mathcal{X}_{2,1}'| &= \frac{1}{2}m(m+1); \\
	\mathcal{X}_{2,2}' &= \bigl\{ x_2 \dots x_{n-2} a_i b_1 y_{n-3} \bigm| 1 \le i \le m \bigr\}, & |\mathcal{X}_{2,2}'| &= m; \\[1mm]
	\mathcal{X}_{2,3}' &= \bigl\{ x_2 \dots x_{n-2} a_m b_m y_{n-3} \bigr\}, & |\mathcal{X}_{2,3}'| &= 1.
\end{align*}
This yields $|\mathcal{X}_2'| = \sum_{t=1}^{3} |\mathcal{X}_{2,t}'| = \displaystyle\frac{1}{2}(m+1)(m+2)$, establishing 
\[
X_2 = X_1 + |\mathcal{X}_2'| = \frac{1}{2}m(m+1) + \frac{1}{2}(m+1)(m+2) = (m+1)^2.
\]
	
\noindent\textit{Computation of $X_3$.} The family $\mathcal{X}_3'$ is represented as $\coprod_{t=1}^{3} \mathcal{X}_{3,t}'$, where
\begin{align*}
	\mathcal{X}_{3,1}' &= \bigl\{ x_3 \dots x_{n-2} a_i b_j l b_q \bigm| 1 \le j \le i \le m-1, \, i+1 \le q \le m \bigr\}, \\[1mm]
	\mathcal{X}_{3,2}' &= \bigl\{ x_3 \dots x_{n-2} a_i b_1 y_{n-3} v \bigm| 1 \le i \le m-1, \, v \in \{b_m, y_{n-4}\} \bigr\}, \\[1mm]
	\mathcal{X}_{3,3}' &= \bigl\{ x_3 \dots x_{n-2} a_m b_1 y_{n-3} y_{n-4}, \, x_3 \dots x_{n-2} a_m b_m y_{n-3} y_{n-4} \bigr\}.
\end{align*}
Clearly, we have $|\mathcal{X}_{3,2}'| = 2m-2$ and $|\mathcal{X}_{3,3}'| = 2$, while  $|\mathcal{X}_{3,1}'|$ is calculated as follows:
\begin{align*}			
	|\mathcal{X}_{3,1}'| = \sum_{i=1}^{m-1} \sum_{j=1}^{i} (m-i) &= \sum_{i=1}^{m-1} i(m-i) = m \sum_{i=1}^{m-1} i - \sum_{i=1}^{m-1} i^2 \\
	&= \frac{1}{2}m^2(m-1) - \frac{1}{6}m(m-1)(2m-1) = \frac{1}{6}m(m^2-1).
\end{align*}
Thus, $|\mathcal{X}_3'| = \sum_{t=1}^{3} |\mathcal{X}_{3,t}'| = \displaystyle\frac{1}{6}m^3 + \frac{11}{6}m$, and we confirm the lemma for $k = 3$:
\begin{align*}
	X_3 = X_2 + |\mathcal{X}_3'| &= (m+1)^2 + \frac{1}{6}m^3 + \frac{11}{6}m \\[1.5mm]
	&= \frac{1}{6}m^3 + m^2 + \frac{23}{6}m + 1.
\end{align*}	
	
\noindent\textit{Computation of $X_k$ for $4 \le k \le n-2$.} 
\smallskip

\noindent For an arbitrary index $4 \le h \le k$, the family $\mathcal{X}_h'$ is partitioned as $\coprod_{t=1}^{3} \mathcal{X}_{h,t}'$, where
\begin{align*}
	\mathcal{X}_{h,1}' &= \bigl\{ x_h \dots x_{n-2} a_i b_1 y_{n-3} \dots y_{n-h-1} \bigm| 1 \le i \le m \bigr\}, \\[1mm]
	\mathcal{X}_{h,2}' &= \bigl\{ x_h \dots x_{n-2} a_m b_m y_{n-3} \dots y_{n-h-1} \bigr\}, \\[1mm] 
	\mathcal{X}_{h,3}' &= \bigl\{ x_h \dots x_{n-2} a_i b_j l b_m y_{n-3} \dots y_{n-h+1} \bigm| 2 \le j \le i \le m-1 \bigr\}. 
\end{align*}
The cardinalities $|\mathcal{X}_{h,1}'| = m$, $|\mathcal{X}_{h,2}'| = 1$, and $|\mathcal{X}_{h,3}'| = \displaystyle\frac{1}{2}(m-1)(m-2)$ sum to
\[
|\mathcal{X}_h'| = \sum_{t=1}^{3} |\mathcal{X}_{h,t}'| = \frac{1}{2}(m^2-m+4).
\]
Using the relation $X_k = X_{k-1} + |\mathcal{X}_k'|$, we arrive at the final expression
\begin{align*}
	X_k = X_3 + \sum_{h=4}^k |\mathcal{X}_h'| &=\frac{1}{6}m^3 + m^2 + \frac{23}{6}m + 1 + \frac{1}{2}(k-3)(m^2-m+4) \\[1.5mm]
	&= \frac{1}{6}m^3 + \frac{1}{2} (k-1)m^2 + \frac{1}{6}(32-3k)m + 2k-5. \qedhere
\end{align*}
\end{proof}

\begin{lemma}\label{l5}
	The $[P_n]$-degrees of the vertices $y_s$ in $G(m,n)$ satisfy
	\begin{align*}
		Y_1 &= \displaystyle\frac{1}{2}m^2 + \displaystyle\frac{9}{2}m - 7, \qquad Y_2 = m^2 + 4m - 5, \\[1.5mm]
		Y_s &= \displaystyle\frac{1}{6}m^3 + \displaystyle\frac{1}{2}(s-1)m^2 + \displaystyle\frac{1}{6}(32-3s)m + 2s - 9 \quad \text{for } 3 \le s \le n-4 \text{ with } n \ge 7, \\[1.5mm]
		Y_{n-3} &= \displaystyle\frac{1}{6}m^3 + \displaystyle\frac{1}{2}(n-4)m^2 + \displaystyle\frac{1}{6}(47-3n)m + 2n - 16.
	\end{align*}
\end{lemma}

\begin{proof}
For each $1 \le s \le n-3$, let $\mathcal{Y}_s'$ be the family of induced paths of order $n$ in $G(m,n)$ with the end-vertex $y_s$. Due to the structural properties of $G(m,n)$, the following recurrence relation holds
\[
Y_1 = |\mathcal{Y}_1'|, \quad Y_s = Y_{s-1} + |\mathcal{Y}_s'| \quad \text{for } 2 \le s \le n-4.
\]

\noindent\textit{Computation of $Y_1$.} We decompose $\mathcal{Y}_1' = \coprod_{t=1}^{4} \mathcal{Y}_{1,t}'$, where
\begin{align*}
	\mathcal{Y}_{1,1}' &= \bigl\{ y_1 \dots y_{n - 3}b_ma_mv \bigm| v \in \{x_{n - 2}\} \cup \{a_1, \dots, a_{m - 1}\} \cup \{b_2, \dots, b_{m - 1}\} \bigr\}, \\[1mm]
	\mathcal{Y}_{1,2}' &= \bigl\{ y_1 \dots y_{n - 3}b_1a_i v \bigm| 1 \le i \le m - 1,\, v \in \{x_{n - 2}\} \cup \{b_2, \dots, b_{i}\} \bigr\}, \\[1mm]
	\mathcal{Y}_{1,3}' &= \bigl\{ y_1 \dots y_{n - 3}b_1a_mv \bigm| v \in \{x_{n - 2}\} \cup \{b_2, \dots, b_{m - 1}\} \bigr\}, \\[1mm]
	\mathcal{Y}_{1,4}' &= \bigl\{ y_1\dots y_{n - 3}vlb_j \bigm| 2 \le j \le m - 1,\, v \in \{b_1, b_m\} \bigr\}.
\end{align*}
The cardinalities of these subfamilies are given by $|\mathcal{Y}_{1,1}'| = 2m-2$, $|\mathcal{Y}_{1,2}'| = \displaystyle\frac{1}{2}m(m - 1)$, $|\mathcal{Y}_{1,3}'| = m-1$, and $|\mathcal{Y}_{1,4}'| = 2m - 4$, which results in
\[
Y_1 = |\mathcal{Y}_1'| = \sum_{t=1}^{4} |\mathcal{Y}_{1,t}'| = \frac{1}{2}m^2 + \frac{9}{2}m - 7.
\]

\noindent\textit{Computation of $Y_2$.} We partition $\mathcal{Y}_2' = \bigsqcup_{t=1}^{3} \mathcal{Y}_{2,t}'$, where
\begin{align*}
	\mathcal{Y}_{2,1}' &= \bigl\{ y_2 \dots y_{n-3} b_1 a_i x_{n-2} x_{n-3} \bigm| 1 \le i \le m \bigr\}, && |\mathcal{Y}_{2,1}'| = m; \\[1mm]
	\mathcal{Y}_{2,2}' &= \bigl\{ y_2 \dots y_{n-3} b_m a_m x_{n-2} x_{n-3} \bigr\}, && |\mathcal{Y}_{2,2}'| = 1; \\
	\mathcal{Y}_{2,3}' &= \bigl\{ y_2 \dots y_{n-3} b_m l b_j a_i \bigm| 2 \le j \le i \le m-1 \bigr\}, && |\mathcal{Y}_{2,3}'|= \frac{1}{2}(m-1)(m-2).
\end{align*}
Consequently, $|\mathcal{Y}_2'| = \sum_{t=1}^{3} |\mathcal{Y}_{2,t}'| =\displaystyle\frac{1}{2}m^2 - \frac{1}{2}m + 2$, so that
\[
Y_2 = Y_1 + |\mathcal{Y}_2'| = \frac{1}{2}m^2 + \frac{9}{2}m - 7 + \frac{1}{2}m^2 - \frac{1}{2}m + 2 = m^2 + 4m - 5.
\]

\noindent\textit{Computation of $|\mathcal{Y}_3'|$.}
The family $\mathcal{Y}_3'$ is split into $\coprod_{t=1}^{3} \mathcal{Y}_{3,t}'$, where
\begin{align*}
	\mathcal{Y}_{3,1}' &= \bigl\{ y_3 \dots y_{n-3} b_1 a_i x_{n-2} x_{n-3} x_{n-4} \bigm| 1 \le i \le m \bigr\}, \\[1mm]
	\mathcal{Y}_{3,2}' &= \bigl\{ y_3 \dots y_{n-3} b_m l b_j a_i v \bigm| 2 \le j \le i \le m-1, \, v \in \{x_{n-2}\} \cup \{a_1, \dots, a_{j-1}\} \bigr\}, \\[1mm]
	\mathcal{Y}_{3,3}' &= \bigl\{ y_3 \dots y_{n-3} b_m a_m x_{n-2} x_{n-3} x_{n-4} \bigr\}.
\end{align*}
For the subfamilies $\mathcal{Y}_{3,1}'$ and $\mathcal{Y}_{3,3}'$, we have $|\mathcal{Y}_{3,1}'| = m$ and $|\mathcal{Y}_{3,3}'| = 1$, while
\begin{align*}	
	|\mathcal{Y}_{3,2}'| = \sum_{j=2}^{m-1} \sum_{i=j}^{m-1} j 
	&= \sum_{j=2}^{m-1} j(m-j) = m \sum_{j=1}^{m-1} j - \sum_{j=1}^{m-1} j^2 - (m-1) \\[1mm]
	&= \frac{1}{2} m^2(m-1) - \frac{1}{6} m(m-1)(2m-1) - (m-1) \\[1.5mm]
	&= \frac{1}{6}(m-1)(m-2)(m+3).
\end{align*}
Summing these cardinalities, we find
\[
|\mathcal{Y}_3'| = \sum_{t=1}^{3} |\mathcal{Y}_{3,t}'| = \frac{1}{6}m^3 - \frac{1}{6}m + 2.
\]

\noindent\textit{Computation of $Y_3$ for $n \ge 7$.} Direct substitution into the recurrence relation yields
\[
Y_3 = Y_2 + |\mathcal{Y}_3'| = m^2 + 4m - 5 + \frac{1}{6}m^3 - \frac{1}{6}m + 2 = \frac{1}{6}m^3 + m^2 + \frac{23}{6}m - 3.
\]
 
\noindent\textit{Computation of $Y_s$ for $4 \le s \le n - 4$, $n \ge 8$.} 
\smallskip

\noindent For an arbitrary index $4 \le r \le s$, the family $\mathcal{Y}_r'$ is represented as $\coprod_{t=1}^{3} \mathcal{Y}_{r,t}'$, where
\begin{align*}
	\mathcal{Y}_{r,1}' &= \bigl\{ y_r \dots y_{n-3} b_1 a_i x_{n-2} \dots x_{n-r-1} \bigm| 1 \le i \le m \bigr\}, \\[1mm]
	\mathcal{Y}_{r,2}' &= \bigl\{ y_r \dots y_{n-3} b_m l b_j a_i x_{n-2} \dots x_{n-r+1} \bigm| 2 \le j \le i \le m-1 \bigr\}, \\[1mm]
	\mathcal{Y}_{r,3}' &= \bigl\{ y_r \dots y_{n-3} b_m a_m x_{n-2} \dots x_{n-r-1} \bigr\}.
\end{align*}
Since $|\mathcal{Y}_{r,1}'| = m$, $|\mathcal{Y}_{r,3}'| = 1$, and $|\mathcal{Y}_{r,2}'| = \displaystyle\frac{1}{2}(m-1)(m-2)$, it follows that
\[
|\mathcal{Y}_r'| = \sum_{t=1}^{3} |\mathcal{Y}_{r,t}'| = \frac{1}{2}(m^2 - m + 4).
\]
By employing the recurrence relation $Y_s = Y_{s-1} + |\mathcal{Y}_s'|$, we obtain
\begin{align*}
	Y_s = Y_3 + \sum_{r=4}^s |\mathcal{Y}_r'| 
	&= \frac{1}{6}m^3 + m^2 + \frac{23}{6}m - 3 + \frac{1}{2}(s-3)(m^2 - m + 4) \\
	&= \frac{1}{6}m^3 + \frac{1}{2}(s-1)m^2 + \frac{1}{6}(32-3s)m + 2s - 9.
\end{align*}

\noindent\textit{Computation of $Y_{n-3}$.}
Let $\mathcal{Y}_{n-3}^*$ be the family of induced paths of order $n$ in $G(m,n)$ containing $y_{n-3}$ but not $y_{n-4}$. We consider two cases depending on the parameter $n$.

\medskip
\noindent\textbf{Case 1: $n \ge 7$.}
Here, $\mathcal{Y}_{n-3}^*$ consists of four disjoint subfamilies
\begin{align*}
	\mathcal{Y}_{n-3,1}^* &= \bigl\{ y_{n-3} b_1 a_i x_{n-2} \dots x_2 \bigm| 1 \le i \le m \bigr\}, \\[1mm]
	\mathcal{Y}_{n-3,2}^* &= \bigl\{ y_{n-3} b_m a_m x_{n-2} \dots x_2 \bigr\}, \\[1mm]
	\mathcal{Y}_{n-3,3}^* &= \bigl\{ y_{n-3} b_m l b_j a_i x_{n-2} \dots x_4 \bigm| 2 \le j \le i \le m-1 \bigr\}, \\[1mm]
	\mathcal{Y}_{n-3,4}^* &= \bigl\{ b_m y_{n-3} b_1 a_i x_{n-2} \dots x_3 \bigm| 1 \le i \le m-1 \bigr\}.
\end{align*}
Since $|\mathcal{Y}_{n-3,1}^*| = m$, $|\mathcal{Y}_{n-3,2}^*| = 1$, $|\mathcal{Y}_{n-3,4}^*| = m-1$, and $|\mathcal{Y}_{n-3,3}^*| = \displaystyle\frac{1}{2}(m-1)(m-2)$, the summation yields
\[
|\mathcal{Y}_{n-3}^*| = \sum_{t=1}^{4} |\mathcal{Y}_{n-3,t}^*| = \frac{1}{2}(m^2 + m + 2).
\]
Due to the structural properties of $G(m,n)$, we have $Y_{n-3} = Y_{n-4} + |\mathcal{Y}_{n-3}^*|$. Therefore,
\begin{align*}
	Y_{n-3} &= \frac{1}{6}m^3 + \frac{1}{2}(n-5)m^2 + \frac{1}{6}(44-3n)m + 2n - 17 + \frac{1}{2}(m^2 + m + 2) \\[1.5mm]
	&= \frac{1}{6}m^3 + \frac{1}{2}(n-4)m^2 + \frac{1}{6}(47-3n)m + 2n - 16.
\end{align*}

\medskip
\noindent\textbf{Case 2: $n = 6$.}
Here, $\mathcal{Y}_3^* = \mathcal{Y}_3' \bigsqcup \bigl\{ b_m y_3 b_1 a_i x_4 x_3 \bigm| 1 \le i \le m-1 \bigr\}$, implying
\[
|\mathcal{Y}_3^*| = |\mathcal{Y}_3'| + m - 1 = \frac{1}{6}m^3 + \frac{5}{6}m + 1.
\]
Since $Y_3 = Y_2 + |\mathcal{Y}_3^*|$, we obtain the final formula for $Y_{n-3}$ in the case $n = 6$:
\[
Y_3 = m^2 + 4m - 5 + \frac{1}{6}m^3 + \frac{5}{6}m + 1 = \frac{1}{6}m^3 + m^2 + \frac{29}{6}m - 4. 
\]
To conclude the proof, we note that the boundary results for $Y_s$ when $s=3$, $n \ge 7$ and $Y_{n-3}$ when $n = 6$ are fully consistent with the general formulas stated in the lemma.
\end{proof}

\begin{lemma} \label{l6}
	The $[P_n]$-degrees of the vertices $a_i$ in $G(m,n)$ are exactly
	\begin{align*}
		A_i &= \frac{1}{2}m^2 - \frac{1}{2}m + (n - 1)i + 2 \quad \text{for } 1 \le i \le m - 1, \\[1mm]
		A_m &= 5m + 2n - 11.
	\end{align*}
\end{lemma}

\begin{proof}
For a fixed $1 \le i \le m-1$, the family $\mathcal{P}_n(G(m,n), a_i)=\coprod_{t=1}^{8} \mathcal{A}_{i,t}$, where 
\begin{align*}
	\mathcal{A}_{i,1} &= \big\{x_1\dots x_{n - 2}a_ib_j, \, x_2\dots x_{n - 2}a_ib_jl \bigm| 1 \le j \le i \big\}, \\[1mm]
	\mathcal{A}_{i,2} &= \big\{x_3\dots x_{n - 2}a_ib_1y_{n - 3}b_m, \ a_ia_mb_my_{n - 3}\dots y_1\big\}, \\[1mm]
	\mathcal{A}_{i,3} &= \big\{x_3\dots x_{n - 2}a_ib_jlb_q \bigm| 1 \le j \le i < q \le m \big\}, \\[1mm]
	\mathcal{A}_{i,4} &= \big\{x_k \dots x_{n - 2}a_ib_1y_{n - 3}\dots y_{n - k - 1} \bigm| 2 \le k \le n - 2 \big\}, \\[1mm]
	\mathcal{A}_{i,5} &= \big\{x_k \dots x_{n - 2}a_ib_jlb_my_{n - 3}\dots y_{n - k + 1} \bigm| 2 \le j \le i, \ 2 \le k \le n - 3 \big\}, \\[1mm]
	\mathcal{A}_{i,6} &= \big\{a_pa_ib_jlb_my_{n - 3}\dots y_3 \bigm| 1 \le p < j \le i \big\}, \\[1mm]
	\mathcal{A}_{i,7} &= \big\{a_ia_pb_jlb_my_{n - 3}\dots y_3 \bigm| i + 1 \le j \le p \le m - 1 \big\}, \\[1mm]
	\mathcal{A}_{i,8} &= \big\{b_ja_ib_1y_{n - 3}\dots y_1 \bigm| 2 \le j \le i \big\}.
\end{align*}
The cardinalities of these subfamilies are given by
\begin{alignat*}{3}	
	|\mathcal{A}_{i,1}| &= 2i,  & |\mathcal{A}_{i,2}| &= 2,  & |\mathcal{A}_{i,3}| &= i(m - i), \\[1mm]
	|\mathcal{A}_{i,4}| &= n - 3, & |\mathcal{A}_{i,5}| &= (n - 4)(i - 1), & |\mathcal{A}_{i,8}| &= i - 1, \\[1mm]
	|\mathcal{A}_{i,6}| &= \frac{1}{2}i(i - 1), \qquad & |\mathcal{A}_{i,7}| &= \frac{1}{2}(m - i)(m - i - 1), &&
\end{alignat*}
which sums to
\[
A_i = \sum_{t=1}^{8} |\mathcal{A}_{i,t}| = \frac{1}{2}m^2 - \frac{1}{2}m + (n - 1)i + 2. 
\]
For the boundary case $i = m$, we have $\mathcal{P}_n(G(m,n), a_m) = \coprod_{t=1}^{4} \mathcal{A}_{m,t}$, where
\begin{align*}
	\mathcal{A}_{m,1} &= \big\{x_1\dots x_{n - 2} a_m b_j, \, x_2\dots x_{n - 2} a_m b_j l \bigm| 1 \le j \le m \big\}, \\[1mm]
	\mathcal{A}_{m,2} &= \big\{x_k\dots x_{n - 2} a_m b_j y_{n - 3}\dots y_{n-k-1} \bigm| 2 \le k \le n - 2, \; j \in \{1, m\}\big\}, \\[1mm]
	\mathcal{A}_{m,3} &= \big\{b_j a_m b_1 y_{n - 3}\dots y_1 \bigm| 2 \le j \le m - 1 \big\}, \\[1mm]
	\mathcal{A}_{m,4} &= \big\{v a_m b_m y_{n - 3}\dots y_1 \bigm| v \in \{a_1, \dots, a_{m - 1}\}\cup\{b_2, \dots, b_{m - 1}\}\big\}.
\end{align*}
Since $|\mathcal{A}_{m,1}| = 2m$, $|\mathcal{A}_{m,2}| = 2n - 6$, $|\mathcal{A}_{m,3}| = m - 2$, and $|\mathcal{A}_{m,4}| = 2m - 3$, we obtain
\[	
A_m = \sum _{t=1} ^4 |\mathcal{A}_{m,t}| = 5m + 2n - 11. \qedhere
\]
\end{proof}

\begin{lemma} \label{l7}
	The \mbox{$[P_n]$-degrees} of the vertices $b_j$ in $G(m,n)$ are given by
	\begin{align*}
		B_1 &= m^2 + mn - 4, \\[1mm]
		B_j &= \frac{1}{2}m^2 + \frac{1}{2}(2n-3)m + 6 - (n-1)j \quad \text{for } 2 \le j \le m-1, \\[1.5mm]
		B_m &= \frac{1}{6}m^3 + \frac{1}{2}(n - 4)m^2 + \frac{1}{6}(59 - 9n)m + 2n - 11.
	\end{align*}
\end{lemma}

\begin{proof}
\noindent\textit{Computation of $B_1$.} The family $\mathcal{P}_n(G(m,n), b_1) = \coprod_{t=1}^{7} \mathcal{B}_{1,t}$, where
\begin{align*}
	\mathcal{B}_{1,1} &= \bigl\{ x_1 \dots x_{n-2} a_i b_1, \, x_3 \dots x_{n-2} a_i b_1 y_{n-3} y_{n-4} \bigm| 1 \le i \le m \bigr\}, \\[1mm]
	\mathcal{B}_{1,2} &= \bigl\{ x_2 \dots x_{n-2} a_i b_1 v \bigm| 1 \le i \le m, \, v \in \{l, y_{n-3}\} \bigr\}, \\[1mm]
	\mathcal{B}_{1,3} &= \bigl\{ x_3 \dots x_{n-2} a_i b_1 l b_q \bigm| 1 \le i < q \le m \bigr\}, \\[1mm]
	\mathcal{B}_{1,4} &= \bigl\{ x_3 \dots x_{n-2} a_i b_1 y_{n-3} b_m \bigm| 1 \le i \le m-1 \bigr\}, \\[1mm]
	\mathcal{B}_{1,5} &= \bigl\{ x_k \dots x_{n-2} a_i b_1 y_{n-3} \dots y_{n-k-1} \bigm| 4 \le k \le n-2, \, 1 \le i \le m \bigr\}, \\[1mm]
	\mathcal{B}_{1,6} &= \bigl\{ b_q l b_1 y_{n-3} \dots y_1 \bigm| 2 \le q \le m-1 \bigr\}, \\[1mm]
	\mathcal{B}_{1,7} &= \bigl\{ b_q a_i b_1 y_{n-3} \dots y_1 \bigm| 2 \le q \le m-1, \, q \le i \le m \bigr\}.
\end{align*}
The cardinalities of these subfamilies are:
\begin{alignat*}{3}			
	|\mathcal{B}_{1,1}| &= |\mathcal{B}_{1,2}| = 2m, \quad & |\mathcal{B}_{1,3}| &= \frac{1}{2}m(m-1), \quad & |\mathcal{B}_{1,4}| &= m - 1, \\[1mm]
	|\mathcal{B}_{1,5}| &= (n-5)m, \quad \quad \, \, & |\mathcal{B}_{1,6}| &= m - 2, && \\[1mm]
	|\mathcal{B}_{1,7}| &= \rlap{$\displaystyle\sum_{q=2}^{m-1} (m-q+1) = \frac{1}{2}(m-2)(m+1).$} &&
\end{alignat*}
Consequently, the \mbox{$[P_n]$-degree} of the vertex $b_1$ in $G(m,n)$ is given by
\[
B_1 = \sum_{t=1}^{7} |\mathcal{B}_{1,t}| = m^2 + mn - 4.
\]

\noindent\textit{Computation of $B_j$ for $2 \le j \le m-1$.} 
\medskip
 
\noindent For a fixed $2 \le j \le m-1$, the family $\mathcal{P}_n(G(m,n), b_j) = \coprod_{t=1}^{7} \mathcal{B}_{j,t}$, where
\begin{align*}
	\mathcal{B}_{j,1} &= \bigl\{ x_1 \dots x_{n-2} a_i b_j, \, x_2 \dots x_{n-2} a_i b_j l, \, b_j a_i b_1 y_{n-3} \dots y_1 \bigm| j \le i \le m \bigr\}, \\[1mm]
	\mathcal{B}_{j,2} &= \bigl\{ x_3 \dots x_{n-2} a_i b_j l b_q \bigm| j \le i < q \le m \bigr\}, \\[1mm]
	\mathcal{B}_{j,3} &= \bigl\{ x_3 \dots x_{n-2} a_i b_q l b_j \bigm| 1 \le i \le j-1, \, 1 \le q \le i \bigr\}, \\[1mm]
	\mathcal{B}_{j,4} &= \bigl\{ x_k \dots x_{n-2} a_i b_j l b_m y_{n-3} \dots y_{n-k+1} \bigm| 4 \le k \le n-2, \, j \le i \le m-1 \bigr\}, \\[1mm]
	\mathcal{B}_{j,5} &= \bigl\{ a_i b_j l b_m y_{n-3} \dots y_2 \bigm| j \le i \le m-1 \bigr\}, \\[1mm]
	\mathcal{B}_{j,6} &= \bigl\{ a_i a_p b_j l b_m y_{n-3} \dots y_3 \bigm| 1 \le i \le j-1, \, j \le p \le m-1 \bigr\}, \\[1mm]
	\mathcal{B}_{j,7} &= \bigl\{ b_j l w y_{n-3} \dots y_1 \bigm| w \in \{b_1, b_m\} \bigr\} \cup \bigl\{ b_j a_m b_m y_{n-3} \dots y_1 \bigr\}.
\end{align*}
The cardinalities of these subfamilies are:
\begin{alignat*}{3}				
	|\mathcal{B}_{j,1}| &= 3(m - j + 1), \quad & |\mathcal{B}_{j,2}| &= \frac{1}{2}(m-j)(m-j+1), \quad & |\mathcal{B}_{j,3}| &= \frac{1}{2}j(j-1), \\[1mm]
	|\mathcal{B}_{j,4}| &= (n-5)(m-j), \quad & |\mathcal{B}_{j,5}| &= m - j, \qquad |\mathcal{B}_{j,7}| = 3, \quad & |\mathcal{B}_{j,6}| &= (j-1)(m-j).
\end{alignat*}
Consequently, the \mbox{$[P_n]$-degree} of the vertex $b_j$ in $G(m,n)$ is given by
\[
B_j = \sum_{t=1}^{7} |\mathcal{B}_{j,t}| = \frac{1}{2}m^2 + \frac{1}{2}(2n-3)m + 6 - (n-1)j.
\]
\noindent\textit{Computation of $B_m$.} The family $\mathcal{P}_n(G(m,n), b_m) = \coprod_{t=1}^{9} \mathcal{B}_{m,t}$, where
\begin{align*}
	\mathcal{B}_{m,1} &= \bigl\{ x_1 \dots x_{n-2} a_m b_m, \, x_2 \dots x_{n-2} a_m b_m l, \, x_2 \dots x_{n-2} a_m b_m y_{n-3}\bigr\}, \\[1mm]
	\mathcal{B}_{m,2} &= \bigl\{ x_3 \dots x_{n-2} a_i b_1 y_{n-3} b_m \bigm| 1 \le i \le m-1 \bigr\}, \\[1mm]
	\mathcal{B}_{m,3} &= \bigl\{ x_3 \dots x_{n-2} a_i b_j l b_m \bigm| 1 \le j \le i \le m-1 \bigr\}, \\[1mm]
	\mathcal{B}_{m,4} &= \bigl\{ x_k \dots x_{n-2} a_m b_m y_{n-3} \dots y_{n-k-1} \bigm| 3 \le k \le n-2 \bigr\}, \\[1mm]
\mathcal{B}_{m,5} &= \bigl\{ x_k \dots x_{n-2} a_i b_j l b_m y_{n-3}\dots y_{n-k+1} \bigm| 4 \le k \le n-2, \, 2 \le j \le i \le m-1 \bigr\}, \\[1mm]
\mathcal{B}_{m,6} &=\bigl \{ a_i a_p b_j l b_m y_{n-3}\dots y_{3} \bigm| 1 \le i < p \le m-1, \, i+1 \le j \le p \bigr\}, \\[1mm]
\mathcal{B}_{m,7} &= \bigl\{ a_i b_j l b_m y_{n-3}\dots y_{2} \bigm| 2 \le j \le i \le m-1 \bigr\}, \\[1mm]
\mathcal{B}_{m,8} &= \bigl\{ v a_m b_m y_{n-3}\dots y_{1} \bigm| v \in \{a_1, \dots, a_{m-1}\} \cup \{b_2, \dots, b_{m-1}\} \bigr\}, \\[1mm]
\mathcal{B}_{m,9} &= \bigl\{ b_j l b_m y_{n-3}\dots y_{1} \bigm| 2 \le j \le m-1 \bigr\}.
\end{align*}
The cardinalities of these subfamilies are:
\begin{alignat*}{3}			
	|\mathcal{B}_{m,1}| &= 3, \qquad & |\mathcal{B}_{m,2}| &= m-1, \qquad & |\mathcal{B}_{m,3}| &= \frac{1}{2}m(m-1), \\[1mm]
	|\mathcal{B}_{m,4}| &= n-4, \quad & |\mathcal{B}_{m,5}| &= \frac{1}{2}(n-5)(m-1)(m-2), \quad & |\mathcal{B}_{m,7}| &= \frac{1}{2}(m-1)(m-2), \\[1mm]
	|\mathcal{B}_{m,8}| &= 2m-3, \quad & |\mathcal{B}_{m,9}| &= m-2, &&
\end{alignat*}
while the remaining value $|\mathcal{B}_{m,6}|$ is calculated as
\begin{align*}		
	|\mathcal{B}_{m,6}| &= \frac{1}{2}\sum_{i=1}^{m-2} (m-i)(m-i-1) \\[1mm]
	&= \frac{1}{2}\sum_{i=1}^{m-1} (m-i)^2 - \frac{1}{2}\sum_{i=1}^{m-1} (m-i) 
	= \frac{1}{2}\sum_{i=1}^{m-1} i^2 - \frac{1}{2}\sum_{i=1}^{m-1} i \\[1mm]
	&= \frac{1}{12}m(m-1)(2m-1) - \frac{1}{4} m(m-1) = \frac{1}{6}m(m-1)(m-2).
\end{align*}
Summing these contributions, we obtain
\[
B_m = \sum_{t=1}^{9} |\mathcal{B}_{m,t}| = \frac{1}{6}m^3 + \frac{1}{2}(n - 4)m^2 + \frac{1}{6}(59 - 9n)m + 2n - 11. \qedhere
\]
\end{proof}	

\begin{lemma} \label{l8}
	The \mbox{$[P_n]$-degree} of the vertex $l$ in $G(m,n)$ equals
	\[
	L = \frac{1}{3}m^3 + \frac{1}{2}(n - 4)m^2 + \frac{1}{6}(52 - 9n)m + n - 8.
	\]
\end{lemma}

\begin{proof}
	The family $\mathcal{P}_n(G(m,n), l)$ decomposes into six disjoint subfamilies
	\begin{align*}
		\mathcal{L}_1 &= \bigl\{ x_2 \dots x_{n-2} a_i b_j l \bigm| 1 \le j \le i \le m \bigr\}, \\[1mm]
		\mathcal{L}_2 &= \bigl\{ x_3 \dots x_{n-2} a_i b_j l b_q \bigm| 1 \le j \le i \le m-1, \, i+1 \le q \le m \bigr\}, \\[1mm]
		\mathcal{L}_3 &= \bigl\{ x_k \dots x_{n-2} a_i b_j l b_m y_{n-3}\dots y_{n-k+1} \bigm| 4 \le k \le n-2, \, 2 \le j \le i \le m-1 \bigr\}, \\[1mm]
		\mathcal{L}_4 &= \bigl\{ a_i a_p b_j l b_m y_{n-3}\dots y_{3} \bigm| 1 \le i < p \le m-1, \, i+1 \le j \le p \bigr\}, \\[1mm]
		\mathcal{L}_5 &= \bigl\{ a_i b_j l b_m y_{n-3}\dots y_{2} \bigm| 2 \le j \le i \le m-1 \bigr\}, \\[1mm]
		\mathcal{L}_6 &= \bigl\{ b_j l u y_{n-3}\dots y_{1} \bigm| 2 \le j \le m-1, \, u \in \{b_1, b_m\} \bigr\}.
	\end{align*}
Recalling that $\mathcal{L}_2 = \mathcal{X}_{3,1}'$ from Lemma~\ref{l4} and $\mathcal{L}_4 = \mathcal{B}_{m,6}$ from Lemma~\ref{l7}, these cardinalities are determined as
\begin{alignat*}{3}					
	|\mathcal{L}_1| &= \frac{1}{2}m(m+1), \quad & |\mathcal{L}_2| &= \frac{1}{6}m(m^2-1), \qquad & |\mathcal{L}_3| &= \frac{1}{2}(n-5)(m-1)(m-2), \\[1mm]
	|\mathcal{L}_6| &= 2(m-2), \quad & |\mathcal{L}_5| &= \frac{1}{2}(m-1)(m-2), \quad  &|\mathcal{L}_4| &= \frac{1}{6}m(m-1)(m-2).
\end{alignat*}
Summing these contributions yields $L$: 
	\[
	L =\sum_{t=1}^{6} |\mathcal{L}_t| = \frac{1}{3}m^3 + \frac{1}{2} (n - 4) m^2 + \frac{1}{6}(52 - 9n)m + n - 8. \qedhere
	\] 
\end{proof}

In the following statements, we compare the $[P_n]$-degrees of vertices in the graph $G(m,n)$ for $n \ge 6$ and $m \ge 2n$, along with additional constraints on the parameter $m$.

\begin{lemma}\label{l9}
	The sequence $\{A_i\}_{i=1}^{m-1}$ is strictly increasing, whereas the sequence $\{B_j\}_{j=2}^{m-1}$ is strictly decreasing. Moreover,
	\[
	A_i \neq B_j \qquad \text{for all } 1 \le i \le m-1 \text{ and } 2 \le j \le m-1.
	\]
\end{lemma}

\begin{proof}
	The strict increase and decrease of these sequences follow from Lemmas~\ref{l6} and \ref{l7}, respectively. Furthermore, for any admissible pair $(i, j)$ and $n \ge 6$, we have the modular relation
	\[
	A_i - B_j = (n - 1)(i + j - m) - 4 \equiv -4 \not\equiv 0 \pmod{n-1},
	\]
	which yields $A_i \neq B_j$.\qedhere
\end{proof}

\begin{lemma}\label{l10}
	The following inequalities hold
	\begin{align*}
		X_2 &> A_i > A_m \qquad \text{for all } 1 \le i \le m-1, \\
		X_2 &> B_j > A_m \qquad \text{for all } 2 \le j \le m-1.
	\end{align*}
\end{lemma}

\begin{proof}
	By the monotonicity of $\{A_i\}_{i=1}^{m-1}$ and $\{B_j\}_{j=2}^{m-1}$ from Lemma~\ref{l9}, it suffices to verify the upper bounds for the maximal terms $A_{m-1}, B_2$ and the lower bounds for the minimal terms $A_1, B_{m-1}$. Since $m \ge 2n \ge 12$, applying the explicit formulas from Lemmas~\ref{l4}, \ref{l6}, and \ref{l7} and bounding $n$ from above by $\dfrac{m}{2}$, we obtain the required inequalities:
	\begin{align*}
		X_2 - A_{m-1} &= \dfrac{1}{2}m^2 + \dfrac{5}{2}m - 1 - (n-1)(m-1) \ge 4m - 2 > 0 \implies X_2 > A_{m-1}, \\[1mm]
		X_2 - B_2     &= \dfrac{1}{2}m^2 + \dfrac{7}{2}m - 7 - n(m-2) \ge \dfrac{1}{2}(9m-14) > 0 \implies X_2 > B_2, \\[1mm]
		A_1 - A_m     &= \dfrac{1}{2}m^2 - \dfrac{11}{2}m + 12 - n \ge \dfrac{1}{2}\bigl((m-6)^2 - 12\bigr) > 0 \implies A_1 > A_m, \\[1mm]
		B_{m-1} - A_m &= \dfrac{1}{2}m^2 - \dfrac{11}{2}m + 16 - n \ge \dfrac{1}{2}\bigl((m-6)^2 - 4\bigr) > 0 \implies B_{m-1} > A_m. \qedhere
	\end{align*}
\end{proof}

\begin{lemma}\label{l11}
	If $n = 7$ with $m \equiv 1 \pmod 6$, or $n \neq 7$ with $m \equiv 3 \pmod{n-1}$, then
	\[
	\{X_1, Y_1\} \cap \{A_i, B_j\} = \emptyset \qquad \text{for all } 1 \le i \le m-1 \text{ and } 2 \le j \le m-1.
	\]
\end{lemma}

\begin{proof}
	For each admissible pair $(i, j)$, we show that any two $[P_n]$-degrees from distinct sets $\{X_1, Y_1\}$ and $\{A_i, B_j\}$ have incongruent residues modulo $n-1$. 
	
	Indeed, by Lemmas~\ref{l4}--\ref{l7}, we have
	\begin{align*}
		X_1 - A_i &= m - 2 - i(n-1) \equiv m - 2 \pmod{n-1}, \\[1mm]
		X_1 - B_j &= 2m - mn - 6 + j(n-1) \equiv m - 6 \pmod{n-1}, \\[1mm]
		Y_1 - A_i &= 5m - 9 - i(n-1) \equiv 5m - 9 \pmod{n-1}, \\[1mm]
		Y_1 - B_j &= 6m - mn - 13 + j(n-1) \equiv 5m - 13 \pmod{n-1}.
	\end{align*}
If $n = 7$ and $m \equiv 1 \pmod 6$, these differences yield the non-zero residues $5, 1, 2, 4 \pmod 6$, respectively. 
If $n \ge 6$ and $n \neq 7$, the condition $m \equiv 3 \pmod{n-1}$ implies that the respective residues are $1, n-4, 6, 2 \pmod{n-1}$, none of which is congruent to $0$. 
Hence, the assertion follows.
\end{proof}

\begin{lemma}\label{l12}
	If $n \ge 7$, then the $[P_n]$-degrees of vertices in the graph $G(m,n)$ satisfy
	\begin{align}
		L > X_{n-2} > Y_{n-3} > X_{n-3} > B_m &> X_{n-4}, \label{l12:eq1}\\[1mm]
		X_i > Y_i &> X_{i-1} \qquad \text{for } n \ge 8 \text{ and } 4 \le i \le n-4, \label{l12:eq2}\\[1mm]
		X_3 > Y_3 > B_1 > Y_2 > X_2 > Y_1 > X_1 &> A_m, \label{l12:eq3}
	\end{align}
	whereas for $n = 6$ the linear order takes the form
	\[
	L > X_4 > Y_3 > X_3 > B_m > B_1 > Y_2 > X_2 > Y_1 > X_1 > A_m.
	\]
\end{lemma}

\begin{proof}
By virtue of Lemmas~\ref{l4}--\ref{l8}, the required inequalities follow from the constraint $m \ge 2n \ge 12$ and the evaluations below. 
		
Chain~\eqref{l12:eq1} is established by
\begin{align*}
	L - X_{n-2}       &= \dfrac{1}{6}\bigl(m^3 - 3m^2 + 14m - 6n(m+1) + 6\bigr) \ge \dfrac{1}{6}(m^3 - 6m^2 + 11m + 6) > 0, \\[1.5mm]
	X_{n-2} - Y_{n-3} &= \dfrac{1}{2}(m^2 - 3m + 14) > 0, \\[1.5mm] 
	Y_{n-3} - X_{n-3} &= m - 5 > 0, \quad
	X_{n-3} - B_m = m(n-3) > 0, \\[1.5mm] 
	B_m - X_{n-4}     &= \dfrac{1}{2}m^2 - mn + \dfrac{5}{2}m + 2 \ge \dfrac{5}{2}m + 2 > 0.
\end{align*}
Inequalities~\eqref{l12:eq2} follow from
\begin{align*}
	X_i - Y_i &= 4 > 0, \qquad Y_i - X_{i-1} = \dfrac{1}{2}(m^2 - m - 4) > 0.
\end{align*}
Chain~\eqref{l12:eq3} is obtained by
\begin{align*}
	X_3 - Y_3 &= 4 > 0, \quad Y_3 - B_1 = \frac{1}{6}m(m^2 + 23 - 6n) + 1 \ge \frac{1}{6}m(m^2 + 23 - 3m) + 1 > 0,  \\[1.5mm]
	B_1 - Y_2 &= m(n-4) + 1 > 0, \quad Y_2 - X_2 = 2m - 6 > 0, \\[1.5mm]
	X_2 - Y_1 &= \dfrac{1}{2}(m^2 - 5m + 16) > 0, \quad Y_1 - X_1 = 4m - 7 > 0, \\[1.5mm]
	X_1 - A_m &= \frac{1}{2}(m^2 - 9m + 22 - 4n) \ge \frac{1}{2}(m^2 - 11m + 22) > 0.
\end{align*}
If $n = 6$, the relations $L > X_4 > Y_3 > X_3 > B_m$ and $B_1 > Y_2 > X_2 > Y_1 > X_1 > A_m$ hold identically to~\eqref{l12:eq1} and~\eqref{l12:eq3} for $n \ge 7$, respectively. The remaining link between $B_m$ and $B_1$ follows from
\[
B_m - B_1 = \frac{1}{6}m(m^2- 31) + 5 > 0. \qedhere
\]		
\end{proof}

\begin{theorem}\label{t6}	
		For all integers $n \ge 6$ and $m \ge 2n$ satisfying either $n = 7$ with $m \equiv 1 \pmod 6$, or $n \neq 7$ with $m \equiv 3 \pmod{n-1}$, the graph $G(m,n)$ is $[P_n]$-irregular.
\end{theorem}

\begin{proof}
	By Lemmas~\ref{l9}--\ref{l12}, all vertex $[P_n]$-degrees in $G(m,n)$ are pairwise distinct. This immediately implies that $G(m,n)$ is $[P_n]$-irregular.
\end{proof}

\begin{corollary} \label{c4} 
	For every integer $n \geq 6$, there exist infinitely many \mbox{$[P_n]$-irregular} graphs. 
\end{corollary} 

By Corollaries~\ref{c1}--\ref{c4}, the assertion of Theorem~\ref{t1} holds.

\section{Conclusion}
This paper provides the first structural results on $[F]$-irregularity. For any path $P_n$ of order $n \ge 3$, we explicitly constructed infinite families of graphs with pairwise distinct vertex $[P_n]$-degrees, and established that the order $k$ of a non-trivial $[P_3]$-irregular graph can be any integer $k \ge 7$ and no other.  
These results prompt us to formulate the following conjecture.

\begin{conjecture}[{Strong Conjecture about $[F]$-irregular graphs}]
	\label{con3}
	For every connected graph $F$ of order $|F| \ge 3$, there exist infinitely many $[F]$-irregular graphs.
\end{conjecture}

A natural next step is to extend the concept of $[F]$-irregularity to oriented graphs, as was recently done for classical $F$-irregularity~\cite{r6}. First and foremost, we pose the existence problem for $[F]$-irregular oriented graphs where $F$ is a directed path or a directed cycle, both of order $n \ge 3$.

\end{document}